\documentclass[12pt]{article}
\usepackage{}
\usepackage{amssymb}
\usepackage{amsmath}
\usepackage{amsthm}
\usepackage{amsfonts}
\usepackage{mathrsfs}
\usepackage{color}
\usepackage{appendix}
\usepackage[linesnumbered,ruled,vlined]{algorithm2e}
\newtheorem{thm}{Theorem}[section]
\newtheorem{prop}{Proposition}[section]
\newtheorem{lem}{Lemma}[section]
\newtheorem{rem}{Remark}[section]

\newtheorem{defn}{Definition}[section]
\newtheorem{exm}{Example}[section]
\newtheorem{ass}{Assumption}[section]

\numberwithin{equation}{section} \allowdisplaybreaks[4]
\begin{document}
\date{}
\pagestyle{plain}
\title{Zero-sum risk-sensitive continuous-time stochastic games with unbounded payoff and transition rates and Borel spaces\thanks{This work was supported in part by the National Natural Science Foundation of China (Grant No. 61673019, 11931018, 62073346, 61573206) and  Guangdong Province Key Laboratory of Computational Science at the Sun Yat-Sen University (2020B1212060032). \emph{(Corresponding author: Li Xia.)}}}
\author{Junyu Zhang, \ Xianping Guo, \ Li Xia \thanks{J. Zhang and X.P. Guo are both with the School of Mathematics, Sun Yat-Sen University, Guangzhou 510275, China (email:
\{mcszhjy;mcsgxp\}@mail.sysu.edu.cn). L. Xia is with the Business
School, Sun Yat-Sen University, Guangzhou 510275, China (email:
xial@tsinghua.edu.cn). All the authors are with the Guangdong
Province Key Laboratory of Computational Science at the Sun Yat-Sen
University. }}
\date{}
\maketitle

{\bf Abstract:} We study a finite-horizon {\it two-person zero-sum
risk-sensitive stochastic game} for continuous-time Markov chains
and Borel state and action spaces, in which payoff rates, transition
rates and terminal reward functions are allowed to be unbounded from
below and from above and the policies can be history-dependent.
Under suitable conditions, we establish the existence of a solution
to the corresponding Shapley equation (SE) by an approximation
technique. Then, by the SE and the extension of the Dynkin's
formula, we prove the existence of a Nash equilibrium and verify
that the value of the stochastic game is the unique solution to the
SE. Moreover, we develop a value iteration-type algorithm for
approaching to the value of the stochastic game. The convergence of
the algorithm is proved by a special contraction operator  in our
risk-sensitive stochastic game. Finally, we demonstrate our main
results by two examples.



\setlength{\baselineskip}{0.25in}

\section{Introduction} \label{intro}

Markov chain is a fundamental model to formulate stochastic dynamic
systems. Markov decision process (MDP) is an important methodology
to study the performance optimization of stochastic dynamic systems
with a single decision-maker \cite{BS96,guo09,Puterman94}. The
classical MDP theory usually focuses on the performance criteria of
the discounted or average rewards. However, the risk-aware
performance criterion is also important for decision-makers in
specific scenarios, such as financial engineering. The
risk-sensitive MDP is motivated by the seminal work of Howard and
Matheson \cite{Howard72} and attracts continuous attention in the
literatures, which involve various risk metrics, such as exponential
utility function \cite{Ghosh14,Guo19}, variance
\cite{Hernandez99,Xia18}, value at risk (VaR) \cite{Hong14},
conditional VaR \cite{Bauerle14,huang16}, percentile
\cite{Filar95,Fu09}, and probability metrics \cite{White88}.

When we further consider stochastic dynamic systems with multiple
decision-makers, game theory is a typical framework. For stochastic
games (also known as Markov games), Markov model is widely adopted since it
has advantages of memoryless property to capture the system
dynamics. The study of stochastic games is pioneered by the seminal
work of Shapely \cite{Shapely53} and can be divided into categories
according to two-person or $n$-person, zero-sum or nonzero-sum,
discrete-time or continuous-time, static or dynamic, and other criteria. A two-person
zero-sum stochastic game is the mostly studied model since the
gains of one player is exactly the loss of the other player. The
Shapley equation is the fundamental concept to study the value of a
stochastic game and to prove the existence of Nash equilibria
\cite{Filar97,Neyman03}.

Most studies on stochastic games focus on the discounted performance
criterion since the associated contract operator is critical to
prove the convergence of games. However, as aforementioned,
risk-related performance metrics are also important for
decision-making in game theory, especially considering the fact that
human players are usually risk-averse in terms of gains but
risk-seeking in terms of losses, as indicated by the prospect theory
\cite{Kahneman79}. Although the risk-sensitive MDPs have been richly
investigated, the study of risk-sensitive stochastic games seems to
be relatively limited in the literatures, in which the risk metric
is usually of the form of exponential utility function
\cite{Howard72,Jacobson73}. That is, let $X$ be the accumulated
discounted rewards, which is a random variable. The targeted risk
metric is set as $\mathbb E[e^{\theta X}]$, instead of $\mathbb
E[X]$ in risk-neutral regime, where $\theta$ is a risk-sensitive
parameter and positive for risk-averse while negative for
risk-seeking. With the Taylor expansion of $\frac{1}{\theta}\ln
\mathbb E[e^{\theta X}]$, we can see that the targeted metric takes
into account both $\mathbb E[X]$ and $D[X]$ simultaneously and is a
proper performance metric for risk-aware decision-making problems
\cite{BU-17,Guo19}.

For risk-sensitive discrete-time stochastic games, Basu and Ghosh
study a two-person zero-sum stochastic game on the infinite horizon
with discounted and ergodic payoff criteria, where the value of
games and equilibria are proved by studying the
Hamilton-Jacobi-Bellman-Isaacs (HJBI) equation \cite{BasuGhosh12}.
They further extend the results to the cases with countable state
space \cite{BasuGhosh14} and nonzero-sum games \cite{BasuGhosh18},
respectively. Moreover,  Ba$\ddot{{\rm u}}$erle and Rieder study a
two-person zero-sum risk-sensitive discrete-time stochastic game
with Borel state and action spaces and bounded rewards, where the
existence of equilibria is proved and the value of the game is
obtained by solving the Shapley equation under continuity and
compactness conditions \cite{BU-17}.

When the stochastic game is of continuous-time,  Ba\c{s}ar studies a
class of nonzero-sum risk-sensitive differential games with
parameterized nonlinear dynamics and parameterized cost functions,
where the sensitivity of the Nash equilibrium is also studied from a
viewpoint of optimal control \cite{Basar99}. Ba\c{s}ar and his
collaborators further extend the risk-sensitive differential game to
a scenario with $n$-person in the model of mean-field, where a
Hamilton-Jacobi-Bellman-Fleming equation is utilized to study the
value of the game \cite{Tembine2011}. There are other excellent
works on risk-sensitive differential games, just to name a few
\cite{Biswas2020,El-Karoui03,Fleming11}, where varied forms of HJBI
equations are studied to prove the equilibrium of games. For the
risk-sensitive stochastic games with continuous-time Markov chains,
Ghosh, Kumarb and Pal \cite{GKC-2016} study infinite-horizon
discounted and ergodic cost risk-sensitive zero-sum stochastic games
for countable Markov chains and bounded transition/cost rates, and
prove the existence of the value and saddle-point equilibrium in the
class of Markov strategies under suitable conditions.  For the case
of unbounded transition rates, Wei studies a zero-sum game with
continuous-time Markov jump processes under the risk-sensitive
finite-horizon criterion with bounded costs \cite{Wei18}. However,
there is no study about risk-sensitive stochastic games with
unbounded payoff rates and general spaces, which is still an open
problem as pointed out by \cite{Wei18}.

In this paper, we study a two-person zero-sum continuous-time
stochastic game with exponential utility function, where the state
and action spaces are Borel and the payoff rate, transition rates
and terminal reward functions are unbounded. We study this game in
the history-dependent policy space. We first establish the
corresponding Shapley equation for this game and study the existence
of the solution with an approximation technique. Then we further
prove the existence of the Nash equilibrium with Markov policies by
utilizing the Shapley equation and the extension of the Dynkin's
formula. We also verify that the value of the game is the unique
solution to the Shapley equation. Moreover, we develop a value
iteration-type algorithm to approach to the solution to the Shapley
equation, by iteratively solving a series of matrix games with
linear programming techniques. Finally, we give two examples to
demonstrate the main results of our paper. The main contribution of
this paper is that it is the first work to study the risk-sensitive
stochastic game with Borel spaces and unbounded transition and
payoff rates, which is the most general model compared with the
literature work like \cite{BU-17, Wei18} and answers the open
problem raised by \cite{Wei18}. Another noteworthy contribution is
that an iteration algorithm is developed in this paper to compute
the Nash equilibrium of the risk-sensitive stochastic game, which is
important while rarely presented in the literature work. We also
guarantee the convergence of the iteration algorithm by a special
contraction operator  in the risk-sensitive stochastic game. Since
the contraction operator in our game model is not a one-step
contraction operator widely used in discounted MDPs, the convergence
proof is not trivial and may shed some light on the computation
issues of risk-sensitive stochastic games.

The rest of the paper is organized as follows. In
Section~\ref{section2}, we introduce the model formulation of
risk-sensitive continuous-time stochastic games (CTSGs). After
giving technical preliminaries in Section~\ref{section3}, we derive
the main results of this paper in Section~\ref{section4}. In Section
\ref{sectional}, we develop a value iteration-type algorithm for
solving the value and the Nash equilibrium of the game and prove the
convergence. In Section~\ref{section5}, we verify the main results
by two examples. Finally, we conclude this paper in
Section~\ref{section6}.

\section{The model of stochastic games}\label{section2}


\textbf{Notation:} For any Borel space $X$ endowed with the Borel
$\sigma$-algebra ${\cal B}(X)$,  we will denote  by  $I_E$ the
indicator function on any subset $E$ of $X$, by $\delta_z(dx)$ the
Dirac measure  at point $z\in X$,  by  $\mathbb{B}_1(X)$ the set of
all bounded Borel measurable functions $u$ on $X$ with the norm
$\|u\|:=\sup_{x\in X}|u(x)|$, and  by $P(X)$ the set of all
probability measures on  ${\cal B}(X)$.

The model of two-person zero-sum risk-sensitive stochastic games for
continuous-time Markov chains is a five-tuple as below.
\begin{equation}
\mathbb{M}:=\{S, (A,A(x)\in {\cal B}(A),x\in S),(B,B(x)\in {\cal B}(B),x\in S), r(x,a,b),   q(dy|x,a,b), {{g(x)}}\}, \label{m-2.1}
\end{equation}
consisting of the following elements:
\begin{itemize}
\item[(a)] a Borel space $S$, called the state space of the games;

\item[(b)]  Borel spaces $A$  and $B$ for players 1 and 2 respectively, called the action spaces of the players in the games;  $A(x)$ and $B(x)$ denote the sets of actions   available to players 1 and 2 respectively  when the system is at state $x\in S$;

\item[(c)] a Borel measurable function $r(x,a,b)$ on $K$, called the payoff rate, where $K:=\{(x,a,b)|x\in S,a\in A(x),b\in B(x)\}$ is assumed in $\mathcal{B}(S\times A\times B);$ that is, $r(x,a,b)$ is the reward rate for player~1 and the cost rate for player~2;

\item[(d)] transition rates $q(dy|x,a,b)$, a Borel  signed
{{measure}} on ${\cal B}(S)$ given $K$, satisfying that $q(D|x,a,b)\geq 0$ for all $D\in {\cal B}(S)$ with $(x,a,b)\in K$  and $x\not\in D$, being conservative in the sense of
$q(S|x,a,b)\equiv 0$, and stable in that of
\begin{eqnarray}
q^{*}(x):=\sup_{a\in A(x),b\in B(x)}q(x,a,b)<\infty,  \ \ \ \
\forall \ x\in S,  \label{Q}
\end{eqnarray}
where $q(x,a,b):=-q(\{x\}|x,a,b)\geq0$ for any $(x,a,b)\in K$;

{\item[(e)] the real-valued terminal reward function $g(x)$ is measurable on $S$.}
\end{itemize}

Next, we give an informal description of the evolution of CTSGs with
model (\ref{m-2.1}).

Roughly speaking, CTSGs evolve as follows: Two players observe
states of a system continuously in time. If the system is at state
$x_t$ at time $t$, player~1 chooses an action $a_t\in A(x_t)$
according to  a given policy, player~2 chooses an action $b_t\in
B(x_t)$ according to a given policy simultaneously, as a consequence
of which, the following happens:

\rm(i)\ An  payoff for player~1 takes place at the rate
$r(x_t,a_t,b_{t})$;
\par
\rm(ii)\ After a random sojourn time (i.e., the holding time at
state $x_t$), the system jumps to a set $D$ ($x_{t}\not\in D$) of
states with the transition probability
$\frac{q(D|x_t,a_t,b_{t})}{q(x_t,a_t,b_{t})}$ determined by the
transition rates $q(dy|x_t,a_t,b_{t})$.  The
distribution function of the sojourn time is
$(1-e^{-\int_t^{t+\delta}q(x_s,a_s,b_s)ds})$, where $\delta$ is the
sojourn time at state $x_t$.

To formalize what is described above, below we describe the
construction of CTSGs under possibly randomized history-dependent
policies.

To construct the process of the underlying dynamic game, we
introduce some notations: Let
$\Omega_0:=(S\times(0,\infty))^{\infty}$,
$\Omega_k:=(S\times(0,\infty))^k\times S\times (\{\infty\}\times
\{\Delta\})^\infty$ for $k\geq 1$ and some $\Delta\not\in S$,
$\Omega:=\cup_{k=0}^\infty\Omega_k$,  $\mathcal{F}$  the Borel
$\sigma$-algebra on the Borel space $\Omega$. Then, we obtain the
measurable space $(\Omega,\mathcal{F})$. For each $k\geq 1$, and
sample
$\omega:=(x_{0},\delta_{1},x_{1},\ldots,\delta_{k},x_{k},\ldots) \in
\Omega$, define
\begin{equation}\label{eq:tk}
T_{k}(\omega):=\delta_{1}+\delta_{2}+\ldots+\delta_{k}, \
T_{\infty}(\omega):=\lim_{k\rightarrow\infty}T_{k}(\omega),  \ {\rm
and}  \  X_k(\omega):=x_k.
\end{equation}
In what follows, the argument $\omega$ is always omitted except some
special informational statements. Then, we define the state process
$\{x_t,t\geq 0\}$ on $(\Omega,\mathcal{F})$ by
\begin{eqnarray}
 x_t:=\sum_{k\geq0}I_{\{T_{k}\leq t<T_{k+1}\}}X_{k}+ I_{\{t\geq T_\infty\}}\Delta
 \label{T_k}, \ \ \ \ {\rm for } \ t\geq 0, \ \ ({\rm with} \ T_{0}:=0).
\end{eqnarray}
Obviously, $x_{t}(\omega)$ is right-continuous on $[0,\infty)$. We
denote $x_{t-}(\omega):=\lim_{s\to t-}x_{s}(\omega)$. Here we have
used the convenience that  $0\times z=0$ and $0+z=z$ for all $z\in
S_\Delta:=S\cup\{\Delta\}$.

For each fixed
$\omega=(x_{0},\delta_{1},x_{1},\ldots,\delta_{k},x_{k},\ldots)\in\Omega$,
from (\ref{eq:tk}), we see that $T_{k}(\omega)$ ($k\geq 1$) denotes
the $k$-th jump moment of $\{x_t,t\geq 0\}$,
$X_{k-1}(\omega)=x_{k-1}$ is the state of the   process on
$[T_{k-1}(\omega),T_k(\omega))$,
$\delta_{k}=T_{k}(\omega)-T_{k-1}(\omega)$ plays the role of sojourn
time at state $x_{k-1}$, and the sample path $\{x_t(\omega), t\geq
0\}$  has at most denumerable states $x_k$, $k=0,1,\ldots$. We do
not intend to consider the controlled process $\{x_t,t\geq 0\}$
after moment $T_\infty$, and thus view it to be absorbed in the
cemetery state $\Delta$. Hence, we write
$A_\Delta:=A\cup\{a_\Delta\}$, $B_\Delta:=B\cup\{a_\Delta\}$,
$A(\Delta):=\{a_\Delta\}$, $B(\Delta):=\{b_\Delta\}$,
$q(\cdot|\Delta,a_\Delta,b_\Delta):\equiv 0$,
$r(\Delta,a_\Delta,b_\Delta):\equiv 0$,  where $a_\Delta$ and
$b_\Delta$ are isolated points.

To precisely define the optimality criterion, we need to introduce
the concept of a policy for each player below, which is an
equivalent expression of that in \cite{GHS2012,GuoABP:2010,PZ2011a}.
\begin{defn} \label{hdp}
A \emph{(history-dependent) policy $\pi_1$} for player~1 is
determined by a sequence $\{\pi_1^k,k\geq 0\}$   such that,  for
$t\geq 0$ and
$\omega=(x_{0},\delta_{1},x_{1},\ldots,\delta_{k},x_{k},\ldots) \in
\Omega$,
\begin{eqnarray*}
\pi_1(da|\omega,t)&=&I_{\{0\}}(t)\pi^{0}_1(da|x_{0},0)+\sum_{k\geq0}I_{\{T_{k}<t\leq T_{k+1} \}}\pi_1^{k}(da|x_{0},\delta_{1},x_{1},\ldots,\delta_{k},x_{k},t-T_{k}) \nonumber \\
&&  + I_{\{t\geq T_{\infty}\}}\delta_{a_\Delta}(da), \label{pi}
\end{eqnarray*}
where  $\pi_1^{0}(da|x_0,0)$ is a stochastic kernel on $A$ given
$S$, $\pi_1^{k} (k\geq 1)$ are stochastic kernels on $A$ given
$(S\times (0,\infty))^{k+1}$, such that
$\pi_1^k(A(x_k)|\cdot)\equiv1$ for all $k\geq 0$.
\end{defn}

A policy $\pi_1( da |\omega,t)$ is called Markovian if the
corresponding kernels $\pi_1^k$ satisfy that
$\pi_1^k(da|x_0,\delta_{1},x_{1},\ldots,\delta_{k},x_{k},t-T_{k})=:\pi_1^k(da|x_k,t)$
(depending only on the current states $x_k$ and time $t$) for all
$k=0,1,\ldots.$ We  denote such a  Markov policy by
$\pi_1(da|x,t)$ for informational implication. Note that $\pi_1( da
|\omega,t)$ and $\pi_1(da|x,t)$ are not time-homogeneous policies.

We denote by $\Pi_1$ and $\Pi_1^{m}$ the sets of all
history-dependent policies $\pi_1$ and Markov ones respectively for
player~1. The corresponding sets $\Pi_2$ and $\Pi_2^m$ of all history-dependent policies  and all
Markov policies  for player~2 are respectively  defined similarly, with
$B(x)$ in lieu of $ A(x)$.

For any initial distribution $\gamma$ on  $S$ and pair of policies
$(\pi_1,\pi_2)\in \Pi_1\times \Pi_2$, as shown by  the Ionescu
Tulcea theorem (e.g., Proposition 7.45 in \cite{BS96}), we see that
there exists a  unique probability measure
$\mathbb{P}_{\gamma}^{\pi_1,\pi_2}$ (depending on $\gamma$ and
$(\pi_1,\pi_2)$)  on $(\Omega,\mathcal{F})$. Let
$\mathbb{E}_\gamma^{\pi_1,\pi_2}$ be the corresponding expectation
operator. In particular, $\mathbb{E}_\gamma^{\pi_1,\pi_2}$ and
$\mathbb{P}_\gamma^{\pi_1,\pi_2}$ are  respectively written as
$\mathbb{E}_x^{\pi_1,\pi_2}$ and $\mathbb{P}_x^{\pi_1,\pi_2}$ when
$\gamma$ is the Dirac measure at an initial state $x$ in $S$.

Fix any finite horizon $T>0$. For each pair of policies
$(\pi_1,\pi_2)\in \Pi_1\times \Pi_2$ and state $x\in S$, we define
the $T$-horizon \emph{risk-sensitive} value $J(\pi_1,\pi_2,0,x)$
of the continuous-time dynamic game by
\begin{eqnarray}
J(\pi_1,\pi_2,0,x):=\mathbb{E}_{\gamma}^{\pi_1,\pi_2}\left[e^{\theta\int_{0}^{T}\int_{A\times
B}r(x_{t},a,b)\pi_1(da|\omega,t)\pi_2(db|\omega,t)dt+{\theta g(x_{T})}}|x_0=x\right],
\label{C-1}
\end{eqnarray}
provided that the integral is well defined, where $\theta$ is a
constant called the risk-sensitive parameter.  In the following
arguments, we assume that $\theta>0$ which indicates a risk-averse
preference. For the other case of $\theta<0$ with risk-seeking
preference, the corresponding results can be obtained with $r$ being replaced by $-r$, and thus the similar arguments are
omitted.

 Note that the process $\{x_t,t\geq 0\}$ on $(\Omega,\mathcal{F}, \mathbb{P}^{\pi_{1},\pi_{2}}_\gamma)$  may {\it not} be Markovian since the  policies $\pi_{1}$ or $\pi_{2}$ can depend on histories
  $(x_{0},\delta_{1},x_{1},\ldots,\delta_{k},x_{k})$. However,
  for each   $\pi_{1} \in \Pi^{m}_{1}$ and $\pi_{2} \in \Pi^{m}_{2}$, it is well known (e.g. \cite{F14}) that  $\{x_t,t\geq 0\}$  is a Markov process on $(\Omega,\mathcal{F}, \mathbb{P}^{\pi_{1},\pi_{2}}_\gamma)$, and thus  for each  $x\in S$ and $t\in [0,T]$, the following expression
\begin{equation}J(\pi_1,\pi_2,t,x):=\mathbb{E}_\gamma^{\pi_1,\pi_2}\left[e^{\theta\int_t^T\int_{A\times B}r(x_s,a,b)\pi_1(da|x_s,s)\pi_2(db|x_s,s)ds+
{\theta g(x_{T})}}|x_t=x\right], \label{eqtx}\end{equation}
for $ \pi_1\in\Pi_1^m,\  \pi_2 \in \Pi_2^m,$
 is  well defined (when the integral exists), and it is called the risk-sensitive value of policy pair $(\pi_{1},\pi_{2})$ from the horizon $t$ to $T$.

 As is well known, the functions on $S$ defined as
\begin{eqnarray*}
 L(x):=\sup_{\pi_1\in \Pi_1} \inf_{\pi_2\in
\Pi_2}J(\pi_1,\pi_2,0,x), \ \   \ {\rm and} \ \ U(x):=\inf_{\pi_2\in \Pi_2}
\sup_{\pi_1\in \Pi_1}J(\pi_1,\pi_2,0,x) \label{3.2}
\end{eqnarray*}
are called the {\it lower value} and the {\it upper value} of the
stochastic game, respectively. It is clear that
\begin{eqnarray*}
L(x)\leq U(x), \ \quad \forall \ x\in S. \label{3.3}
\end{eqnarray*}

\begin{defn}   If $L(x)= U(x)$ for all $x\in S$, then the
common function is called the {\bf value function} of the game model
$\mathbb{M}$ and is denoted by $\mathbb{M}(x)$.
\end{defn}

 \begin{defn}   Suppose that the game $\mathbb{M}$ has a value function
$\mathbb{M}(x)$. Then a policy $\pi_1^*$ in $\Pi_1$ is said
 to be optimal for player~1 if
 \begin{eqnarray*}
 \inf_{\pi_2 \in \Pi_2} J(\pi_1^*,\pi_2,0,x)=\mathbb{M}(x), \ \ \ \
 \forall \ x\in S. \label{4.3}
\end{eqnarray*}
Similarly, $\pi_2^* \in \Pi_2$ is optimal for player~2 if
\begin{eqnarray*}
 \sup_{\pi_1 \in \Pi_1} J(\pi_1,\pi_2^*,0,x)=\mathbb{M}(x), \ \ \ \
 \forall \ x\in S. \label{4.4}
\end{eqnarray*}
If $\pi_k^*\in \Pi_k$ is optimal for player $k$, $k=1,2$, then
$(\pi_1^*,\pi_2^*)$ is called a Nash equilibrium of the game.
\end{defn}
The aim of this paper is to give conditions for the existence and
the computation of a Nash equilibrium.

\section{Preliminaries}\label{section3}

This section provides  some preliminary facts for our arguments
below. Since the transition rate $q(dy|x,a,b)$ and payoff rate
$r(x,a,b)$ are allowed to be unbounded, we next give conditions for
the non-explosion of $\{x_t,t\geq 0\}$ and finiteness of
$J(\pi_1,\pi_2,0,x)$.

\begin{ass}\label{ass:3.1}
{\rm There exist a real-valued Borel measurable function $V_0(x)\geq 1$ on $S$ and  positive constants
$\rho_0,  L_0, M_0$, such that
\begin{description}
\item[(i)] \  $\int_{S}V_0(y)q(dy|x,a,b)\leq \rho_0 V_0(x)$ \emph{\rm {for all}} $(x,a,b)\in K$;

\item[(ii)] $q^*(x)\leq L_0V_0(x)$ for all $x\in S$, where $q^*(x)$ is as in (\ref{Q});

\item[(iii)]  $|r(x,a,b)|\leq M_0+\frac{\sqrt{2}}{2}\sqrt{\ln V_0(x)}$ for  all $(x,a,b)\in K$ and $|g(x)|\leq M_0+\frac{\sqrt{2}}{2}\sqrt{\ln V_0(x)} $
for  all $x\in S$.

\end{description}
}
\end{ass}


\begin{lem}\label{lem3.1} \rm Under Assumption  \ref{ass:3.1} (i,ii), for each $(\pi_1,\pi_2)\in\Pi_1\times\Pi_2$, $x\in S, t\geq 0$, the following assertions hold.
\begin{description}
\item[(a)]  $\mathbb{P}_{x}^{\pi_1,\pi_2}(x_{t}\in S)=1$, $\mathbb{P}_{x}^{\pi_1,\pi_2}(T_\infty=\infty)=1$, and $ \mathbb{P}_{x}^{\pi_1,\pi_2}(x_{0}=x)=1$.
\item[(b)] $\mathbb{E}_{x}^{\pi_1,\pi_2}[V_0(x_t)] \leq e^{\rho_0t}V_0(x)$, and $\mathbb{E}_{\gamma}^{\pi_1,\pi_2}[V_0(x_t)\big|x_s=x] \leq e^{\rho_0(t-s)}V_0(x)$  for  $t\geq s\geq 0$ when $(\pi_1,\pi_2)$ is in $\Pi^m_1\times\Pi_2^m$.
\item[(c)] If, in addition,  Assumption  \ref{ass:3.1} (iii) is satisfied, then for each $t\geq 0$
 \begin{description}
 \item[(c$_1$)]  $   e^{ -\theta \left[ T e^{\rho_{0}T}+M_{0}T+ e^{\rho_{0}T}+M_{0} \right]V_{0}(x) }\leq J(\pi_1,\pi_2,0,x)\leq LV_0(x)$,\\ where {$L:= e^{2T\theta(M_0+T\theta)+2\theta(M_0+\theta)+\rho_0 T};$}
\item[(c$_2$)]  $  e^{ -\theta \left[ T e^{\rho_{0}T}+M_{0}T+ e^{\rho_{0}T}+M_{0} \right]V_{0}(x) }\leq J(\pi_1,\pi_2,t,x)\leq LV_0(x)$ for $(\pi_1,\pi_2)\in \Pi^m_1\times\Pi_2^m$.
\end{description}
 \end{description}
\end{lem}
\begin{proof}   Parts (a) and (b) follow from  any reference of  \cite{G07,GHS2012,GuoABP:2010,GuoZhang20,PZ2011a}.
We next prove  part (c).  Since $|r(x,a,b)|\leq
\frac{\sqrt{2}}{2}\sqrt{\ln V_0(x)}+M_0\leq T\theta+\frac{\ln
\sqrt{V_0(x)}}{2T\theta}+M_0$ and $|g(x)|\leq
\frac{\sqrt{2}}{2}\sqrt{\ln V_0(x)}+M_0\leq \theta+\frac{\ln
\sqrt{V_0(x)}}{2\theta}+M_0$, we have $e^{T\theta|r(x,a,b)|}\leq
e^{2T\theta|r(x,a,b)|}\leq e^{2T\theta(M_0+T\theta)}\sqrt{V_0(x)}$
and $e^{\theta|g(x)|}\leq e^{2\theta|g(x)|}\leq
e^{2\theta(M_0+\theta)}\sqrt{V_0(x)}$. Using the Jensen inequality
with respect to the probability measure $\frac{dt}{T}$ on
${\cal{B}}([0,T])$ and Cauchy-Buniakowsky-Schwarz Inequality, by (\ref{C-1}) we have
 \begin{eqnarray}\label{eq:3.1}
&&\mathbb{E}_{x}^{\pi_1,\pi_2}\left[e^{ \theta \int_0^T\int_{A\times B} r(x_t,a,b)\pi_1(da|\omega,t)\pi_2(db|\omega,t)dt+\theta g(x_{T})}\right]\nonumber\\
&\le &\mathbb{E}_{x}^{\pi_1,\pi_2}\left[e^{ \theta \int_0^T\int_{A\times B} |r(x_t,a,b)|\pi_1(da|\omega,t)\pi_2(db|\omega,t)dt+\theta |g(x_{T})|}\right]\nonumber\\
&=&\mathbb{E}_{x}^{\pi_1,\pi_2}\left\{e^{ \theta \int_0^T\left[\int_{A\times B} |r(x_t,a,b)|\pi_1(da|\omega,t)\pi_2(db|\omega,t)+\frac{1}{T}|g(x_{T})|\right]dt}\right\}\nonumber\\
&=&\mathbb{E}_{x}^{\pi_1,\pi_2}\left\{e^{ \theta T \int_0^T
\frac{1}{T}\left[\int_{A\times B} |r(x_t,a,b)|\pi_1(da|\omega,t)\pi_2(db|\omega,t)+\frac{1}{T}|g(x_{T})|\right]dt}\right\}\nonumber\\
&\leq& \mathbb{E}_{x}^{\pi_1,\pi_2}\left\{\frac{1}{T} \int_0^T e^{T\theta\left[\int_{A\times B} |r(x_t,a,b)|\pi_1(da|\omega,t)\pi_2(db|\omega,t)+\frac{1}{T}|g(x_{T})|\right]}dt\right\} \nonumber \\
&\leq& \mathbb{E}_{x}^{\pi_1,\pi_2}\left\{\frac{1}{T} \int_0^T e^{2T\theta(M_0+T\theta)}\sqrt{V_0(x_{t})}e^{2\theta(M_0+\theta)}\sqrt{V_0(x_{T})}dt\right\} \nonumber \\
 &=& \frac{1}{T} e^{2T\theta(M_0+T\theta)+2\theta(M_0+\theta)}\int_0^T\mathbb{E}_{x}^{\pi_1,\pi_2}\left[\sqrt{V_0(x_{t})}\sqrt{V_0(x_{T})}\right]dt \nonumber \\
  &\le & \frac{1}{T} e^{2T\theta(M_0+T\theta)+2\theta(M_0+\theta)}\int_0^T\sqrt{\mathbb{E}_{x}^{\pi_1,\pi_2}V_0(x_{t})\mathbb{E}_{x}^{\pi_1,\pi_2}V_0(x_{T})}dt \nonumber \\
  &\le & \frac{1}{T} e^{2T\theta(M_0+T\theta)+2\theta(M_0+\theta)}\int_0^T\sqrt{e^{\rho_{0}t}V_0(x)e^{\rho_{0}T}V_0(x)}dt \nonumber \\
   &\le & \frac{1}{T} e^{2T\theta(M_0+T\theta)+2\theta(M_0+\theta)}\int_0^T e^{\rho_{0}T}V_0(x)dt \nonumber \\
 &\leq&   e^{2T\theta(M_0+T\theta)+2\theta(M_0+\theta)+\rho_0 T}V_0(x).  \label{eq:3.4}
 \end{eqnarray}

  On the other hand, 
 we have
  \begin{eqnarray*}
&&\mathbb{E}_{x}^{\pi_1,\pi_2}\left[e^{ \theta \int_0^T\int_{A\times B} r(x_t,a,b)\pi_1(da|\omega,t)\pi_2(db|\omega,t)dt+\theta g(x_{T})}\right]\nonumber\\
&\ge&e^{\mathbb{E}_{x}^{\pi_1,\pi_2} \left[\theta \int_0^T\int_{A\times B} r(x_t,a,b)\pi_1(da|\omega,t)\pi_2(db|\omega,t)dt+\theta g(x_{T})\right]}\nonumber\\
&\ge&e^{\mathbb{E}_{x}^{\pi_1,\pi_2} \left[-\theta \int_0^T\int_{A\times B} |r(x_t,a,b)|\pi_1(da|\omega,t)\pi_2(db|\omega,t)dt-\theta |g(x_{T})|\right]}\nonumber\\
&\ge&e^{\mathbb{E}_{x}^{\pi_1,\pi_2} \left[-\theta \int_0^T\int_{A\times B} (\ln\sqrt{V_{0}(x_{t})}+M_{0})\pi_1(da|\omega,t)\pi_2(db|\omega,t)dt-\theta (\ln\sqrt{V_{0}(x_{T})}+M_{0})\right]}\nonumber\\
&\ge&e^{\mathbb{E}_{x}^{\pi_1,\pi_2} \left[-\theta \int_0^T (\sqrt{V_{0}(x_{t})}+M_{0})dt-\theta (\sqrt{V_{0}(x_{T})}+M_{0})\right]}\nonumber\\
&\ge&e^{\mathbb{E}_{x}^{\pi_1,\pi_2} \left[-\theta \int_0^T(V_{0}(x_{t})+M_{0})dt-\theta (V_{0}(x_{T})+M_{0})\right]}\nonumber\\
&=&e^{ \left[-\theta \int_0^T(\mathbb{E}_{x}^{\pi_1,\pi_2}V_{0}(x_{t})+M_{0})dt-\theta (\mathbb{E}_{x}^{\pi_1,\pi_2}V_{0}(x_{T})+M_{0})\right]}\nonumber\\
&\ge &e^{ \left[-\theta \int_0^T(e^{\rho_{0}t}V_{0}(x)+M_{0})dt-\theta (e^{\rho_{0}T}V_{0}(x)+M_{0})\right]}\nonumber\\
&\ge&e^{ -\theta\left[ T e^{\rho_{0}T}V_{0}(x)+M_{0}T+ e^{\rho_{0}T}V_{0}(x)+M_{0} \right] }\nonumber\\
&\ge&e^{ -\theta \left[ T e^{\rho_{0}T}+M_{0}T+ e^{\rho_{0}T}+M_{0} \right]V_{0}(x) },
 \end{eqnarray*}
which, together with (\ref{eq:3.4}), implies (c$_1$).  Similarly, we
see that (c$_2$) is also true.
\end{proof}

\begin{rem} \label{rln}It can be seen from the proof, when $g(x)=0$ on $S$ and item (iii) in Assumption \ref{ass:3.1} is weaken to  $|r(x,a,b)|\leq M_0+\frac{1}{2T\theta}\ln V_0(x)$,
item (c) in Lemma \ref{lem3.1} still holds.
\end{rem}

Lemma \ref{lem3.1} gives conditions for the  finiteness of
$J(\pi_1,\pi_2,t,x)$ as well as the non-explosion of $\{x_t,t\geq
0\}$.  In order to deal with the game problem  for history-dependent
policies, we need the {\it analog} of the Ito-Dynkin's formula in
\cite{GHH2015} for possible non-Markov processes $\{x_t,t\geq 0\}$
and functions  $\varphi(\omega,t,x)$ with an  additional element
$\omega\in \Omega$. To do so, we recall some   concepts. Take  the
right-continuous family of $\sigma$-algebras $\{{\cal F}_t\}_{t\ge
0}$ with ${\cal F}_t:=\sigma(\{T_k\leq s, X_k\in D\}: D\in
{\cal{B}}(S),s\leq t,k\geq 0)$. As in
\cite{GuoABP:2010,GS11,PZ2011a}, let $\cal P$ be the
$\sigma$-algebra of predictable sets on $\Omega\times[0,\infty)$
related to $\{{\cal F}_t\}_{t\ge 0}$, that is,
${\cal{P}}:=\sigma(B\times [0,\infty), C\times (s,\infty): B\in
{\cal{F}}_0, C\in {\cal{F}}_{s-}, s>0)$ with
${\cal{F}}_{s-}:=\bigvee_{t<s}{\cal{F}}_t:=\sigma({\cal{F}}_{t},t<s)$.
A real-valued function on $\Omega\times  [0,\infty)$ is called {\it
predictable} if it is measurable with respect to $\cal P$.

Denote by $m_{L}$ the Lebesgue's measure on $[0,T]$, and by
$\mathbb{B}_{\cal P}(\Omega\times [0,T]\times S)$ the set of
real-valued and ${\cal{P}}\times {\cal{B}}(S)$-measurable functions
$\varphi$  with the following features:   Given any $x\in S,
(\pi_1,\pi_2)\in\Pi_1\times\Pi_2$, and a.s. $\omega\in \Omega$ with
respect to $\mathbb{P}_{x}^{\pi_1,\pi_2}$, there exists a Borel
subset $E_{(\varphi,\omega,x,\pi_1,\pi_2)}$ (depending on the
$\varphi,\omega,x,(\pi_1,\pi_2)$) of $[0,T]$ such that the partial
derivative $\frac{\partial \varphi(\omega,t,x) }{\partial t}$  with respect to $t$
exists for every $t\in E_{(\varphi,\omega,x,\pi_1,\pi_2)}$ and
$m_{L}(E_{(\varphi,\omega,x,\pi_1,\pi_2)}^c)=0$. Obviously, if a
function $\varphi(\omega,t,x)$ in $\mathbb{B}_{\cal P}(\Omega\times
[0,T]\times S)$ is independent of $\omega$ (written as
$\varphi(t,x)$), then  the corresponding
$E_{(\varphi,\omega,x,\pi_1,\pi_2)}$ is independent of
$(\omega,\pi_1,\pi_2)$, which will be denoted by   $E_{(\varphi,x)}$
for simplicity.

Next we state the extension of the Dynkin's formula by Lemma
\ref{Th-3.2}. To do so, we introduce the following conditions and
notations.

\begin{ass}\label{ass:3.1b}
{\rm There exist a real-valued Borel measurable function  $V_1(x)  \geq 1$ on $S$, and positive  constants
$\rho_1, b_1$,  and  $M_1$,  such that
\begin{description}
\item[(i)] \  $\int_{S}V_1^2(y)q(dy|x,a,b)\leq \rho_1V_1^2(x)+b_1$ \emph{\rm {for all}} $(x,a,b)\in K$;

\item[(ii)] $V_0^2(x)\leq M_1V_1(x)$ for all $x\in S$, with $V_0(x)$ satisfying Assumption \ref{ass:3.1}.
\end{description}
}
\end{ass}

Assumption \ref{ass:3.1b} is used to give a domain for the Dynkin's formula below, and it is obviously satisfied when the transition rates are bounded  \cite{Ghosh14, KC-13, KC-15,  Yu77}.

 Given the $V_k (k=0,1)$  as in   Assumption \ref{ass:3.1b} and any Borel set $Z$, a real-valued function $\varphi$ on $Z\times S$ is called   $V_k$-bounded if the $V_k$-weighted norm
 of $\varphi$, $\|\varphi\|_{V_k}:=\sup_{(z,x)\in Z\times S}\frac{|\varphi(z,x)|}{V_k(x)}$, is finite. We denote by $\mathbb{B}_{V_k}(Z\times S)$ the Banach space
 of all $V_k$-bounded   functions on $Z\times S$. When $V_{k}(x)\equiv 1$ for all $x\in S$, $\mathbb{B}_{1}(Z\times S)$ is the space of all bounded functions. In particular, take $Z=\Omega\times [0,T]$ or $[0,T]$, we  define
 $$  \mathbb{B}_{V_0,V_1}^{1}(\Omega\times [0,T]\times S):=\left\{\varphi\in \mathbb{B}_{V_0}(\Omega\times [0,T]\times S)\cap\mathbb{B}_{\cal{P}}(\Omega\times [0,T]\times S)\mid \frac{\partial \varphi }{\partial t}\in \mathbb{B}_{V_1}(\Omega\times [0,T]\times S)\right\},$$
 and then
\begin{equation}  \mathbb{B}_{V_0,V_1}^{1}([0,T]\times S):=\left\{\varphi\in \mathbb{B}_{V_0}([0,T]\times S)\cap\mathbb{B}_{\cal{P}}( [0,T]\times S)\mid  \frac{\partial \varphi }{\partial t}\in \mathbb{B}_{V_1}([0,T]\times S)\right\}.\label{eq:B}\end{equation}

\begin{lem}\label{Th-3.2}
{\rm Suppose Assumptions \ref{ass:3.1} and \ref{ass:3.1b}   are satisfied. Then, for each  $(s,x)\in [0,T]\times S$, the following assertions  hold.
\begin{description}
\item[(a)] (The extension of the Dynkin's formula):  For every $(\pi_1,\pi_2)\in\Pi_1\times\Pi_2$ and ${\varphi}\in \mathbb{B}_{V_0,V_1}^{1}(\Omega\times[0,T]\times S)$,
\begin{eqnarray*}
 && \mathbb{E}_{x}^{\pi_1,\pi_2}\left[\int_{0}^{T}
\left(\varphi'(\omega,t,x_{t})+\int_{S}\int_{A\times B}{\varphi}(\omega,t,y)q(dy|x_t,a,b)\pi_1(da|\omega,t)\pi_2(db|\omega,t)\right)dt
\right] \\&& \ \ =\mathbb{E}_{x}^{\pi_1,\pi_2}\left[{\varphi}(\omega,T,x_{T})\right]-\mathbb{E}_{x}^{\pi_1,\pi_2}\left[{\varphi}(\omega,0,x)\right],
\end{eqnarray*}
where
$\{x_t,t\geq 0\}$ may be not Markovian since  policies $\pi_1$ and $\pi_2$ may depend on histories.
\item[(b)] (The Dynkin's formula):  For each $(\pi_1,\pi_2)\in\Pi_1^m\times \Pi_2^m$, and ${\varphi}\in \mathbb{B}_{V_0,V_1}^{1}([0,T]\times S)$,
\begin{eqnarray*}
&&\mathbb{E}_{\gamma}^{\pi_1,\pi_2}\left\{\int_{s}^{T}
\left[\left(e^{\int_s^t\theta
r(x_v,\pi_1^v,\pi_2^v)dv{}}{\varphi}(t,x_t)\right)'\right.\right.\\
&&\left.\left.+\int_Se^{\int_s^t\theta
r(x_v,\pi_1^v,\pi_2^v)dv{}}{\varphi}(t,y)q(dy|x_t,\pi_1^t,\pi_2^t)\right]dt\Big|x_s=x\right\}\\
&=&\mathbb{E}_{\gamma}^{\pi_1,\pi_2}\left[e^{\int_s^T \theta
r(x_t,\pi_1^t,\pi_2^t)dt{}}{\varphi}(T,x_{T})\Big|x_s=x\right]-{\varphi}(s,x),
\ \ \ \forall \ s\in [0,T],
\end{eqnarray*}
where, the notation $u(x,\pi_1^t,\pi_2^t)$
means that
\begin{equation*}
u(x,\pi_1^t,\pi_2^t):=\int_{A(x)}\int_{B(x)}
u(x,a,b)\pi_1(da|x,t)\pi_2(db|x,t), \ \ \ \forall \ x\in S, t\geq
0.\end{equation*}
\end{description}

}
\end{lem}

\begin{proof} See Theorem 3.1 in \cite{GuoZhang20}.
\end{proof}

\section{The existence of Nash equilibria}\label{section4}

In this section, we prove   the existence of  a Nash equilibrium and
a solution to the following Shapley equation (\ref{E-2})  for the
finite-horizon stochastic game with the risk-sensitive criterion.
The proofs are shown in three steps as follows:  1)   consider the
case of bounded transition and payoff rates, 2) deal with the case
of unbounded transition rates but nonnegative payoff, and 3) study
the case of unbounded transition   and   payoff rates.

 \begin{ass}\label{ass:3.3}
{\rm
\begin{description}
\item[(i)] For each $x\in S$, $A(x)$ and $B(x)$ are   compact;
\item[(ii)] For each $x\in S$ and $D\in {\cal{B}}(S)$, the function  $q(D|x,a,b)$  is continuous in $(a,b)\in A(x)\times B(x)$;

 \item[(iii)] For each $x\in S$, the functions $r(x,a,b)$ and  $\int_{S}V_0(y)q(dy|x,a,b)$ are  continuous in $(a,b)\in A(x)\times B(x)$,  with  $V_0(x)$ as in Assumption \ref{ass:3.1}.
 \end{description}
}
\end{ass}
The following results are for the case of the bounded transition and
bounded payoff rates.
\begin{prop}\label{P-4.1} {\rm Under Assumption \ref{ass:3.3} suppose that  $\|q\|:=\sup_{x\in S} q^*(x)$,\\ $\|r\|:=\sup_{x\in S,a\in A(x),b\in B(x) }|r(x,a,b)|$  and $\|g\|:=\sup_{x\in S}|g(x)|$ are finite. Then, the following assertions   hold.
\begin{description}
\item[(a)]   There exists a unique {$ \varphi$} in $\mathbb{B}^{1}_{1,1}([0,T]\times S)$ (that is, $V_{0}(x)=V_{1}(x)\equiv 1, x\in S$ in $(\ref{eq:B}$))  satisfying  the following {\it Shapley equation} for the risk-sensitive criterion of CTSGs on the finite horizon:
{
\begin{eqnarray}
   \begin{cases}
 \displaystyle   {\varphi}'(t,x)+\sup_{\phi\in P(A(x))}\inf_{\psi\in P(B(x))}\left[\theta r(x,\phi,\psi){\varphi}(t,x)+\int_{S}{\varphi}(t,y)q(dy|x,\phi,\psi)\right]=0, \\
    {\varphi}(T,x)=e^{\theta g(x)}, \label{E-2}
    \end{cases}
 \end{eqnarray}
for each $x\in S$ and   $t\in E_{({\varphi},x)}$ with $m_L(E_{({\varphi},x)}^c)=0$.}

\item[(b)] There exists a pair  of Markov policies  $(\hat\pi_1,\hat\pi_2)\in \Pi_1^m\times \Pi_2^m$ such that, for   $x\in S,t\in [0,T]$,
    \begin{eqnarray}
     -{\varphi}'(t,x)&=&\theta r(x,\hat\pi^t_1,\hat\pi^t_2){\varphi}(t,x)+\int_{S}{\varphi}(t,y)q(dy|x,\hat\pi^t_1,\hat\pi^t_2)  \label{S-1}\\
                    &=& \sup_{\phi\in P(A(x))}\left[\theta r(x,\phi_,\hat\pi^t_2){\varphi}(t,x)+\int_{S}{\varphi}(t,y)q(dy|x,\phi,\hat\pi^t_2)\right] \label{S-2}\\
                     &=& \inf_{\psi\in P(B(x))}\left[\theta r(x,\hat\pi^t_1,\psi){\varphi}(t,x)+\int_{S}{\varphi}(t,y)q(dy|x,\hat\pi^t_1,\psi)\right]. \label{S-3}
      \end{eqnarray}
\item[(c)] $\displaystyle{\varphi}(t,x)= \sup_{\pi_{1}\in \Pi_1^m}\inf_{\pi_2\in\Pi_2^m}J(\pi_1,\pi_2,t,x)=\inf_{\pi_2\in\Pi_2^m}\sup_{\pi_{1}\in \Pi_1^m}J(\pi_1,\pi_2,t,x)=J(\hat\pi_1,\hat\pi_2,t,x)$,\\  for  $(t,x)\in [0,T]\times S$.
\item[(d)] The value function $\mathbb{M}(x)$ exists and is equal to ${\varphi}(0,x)$ for each $x\in S$, and $(\hat\pi_1,\hat\pi_2)$ is a Nash equilibrium.

\item[(e)] If $r(x,a,b)\ge 0$ for all $(x,a,b)\in K$, then ${\varphi}(t,x)$ is decreasing in $t\in [0,T]$ for any given $x\in S$, which means that
${\varphi}'(t,x)\le 0, a.e.$.
  \end{description}

}
\end{prop}
The proof is rather long and is therefore presented in Appendix A.
\begin{rem}\label{remde}
{\rm   Even $r(x,a,b)\ge 0$ for all $(x,a,b)\in K$, it is not obvious that ${\varphi}(t,x)$ is decreasing in $t\in [0,T]$ for any given $x\in S$. Although
\begin{equation*}J(\pi_1,\pi_2,t,x)=\mathbb{E}_\gamma^{\pi_1,\pi_2}\left[e^{\theta\int_t^T\int_{A\times B}r(x_v,a,b)\pi_1(da|x_v,v)\pi_2(db|x_v,v)dv+
{\theta g(x_{T})}}|x_t=x\right]\end{equation*} and
\begin{equation*}J(\pi_1,\pi_2,s,x)=\mathbb{E}_\gamma^{\pi_1,\pi_2}\left[e^{\theta\int_s^T\int_{A\times
B}r(x_v,a,b)\pi_1(da|x_v,v)\pi_2(db|x_v,v)dv+ {\theta
g(x_{T})}}|x_s=x\right]\end{equation*} have the same initial state
and policies, their initial times $t$ and $s$ are different.
Therefore, it is not easy to compare $J(\pi_1,\pi_2,t,x)$ and
$J(\pi_1,\pi_2,s,x)$.
 }
\end{rem}

Proposition \ref{P-4.1} shows the existence of a Nash equilibrium under the bounded transition and payoff  rates.   We next extend the results in Proposition \ref{P-4.1} to the case of
unbounded transition rates  and nonnegative payoff rates by approximations.

\begin{prop}\label{Th-4.2} {\rm Under Assumptions \ref{ass:3.1}, \ref{ass:3.1b} and \ref{ass:3.3}, if in addition  $r(x,a,b)\geq 0$  for all $(x,a,b)\in K$
and  {$g(x)\geq 0$  for all $x\in S$}, then the following assertions  hold.

 \begin{description}
\item[(a)]   There exists a unique ${\varphi}$ in $\mathbb{B}^{1}_{V_0,V_1}([0,T]\times S)$ satisfying the following {\it Shapley equation} for the risk-sensitive criterion of CTSGs on the finite horizon:
\begin{eqnarray}
   \begin{cases}
   \displaystyle {\varphi}'(t,x)+\sup_{\phi\in P(A(x))}\inf_{\psi\in P(B(x))}\left[\theta r(x,\phi,\psi){\varphi}(t,x)+\int_{S}{\varphi}(t,y)q(dy|x,\phi,\psi)\right]=0, \\
    {\varphi}(T,x)={e^{\theta g(x)}}, \label{EEE-2}
    \end{cases}
 \end{eqnarray}
for each $x\in S$ and   $t\in E_{({\varphi},x)}$ with $m_L(E_{({\varphi},x)}^c)=0$.
\item[(b)]  $\displaystyle{\varphi}(t,x)= \sup_{\pi_{1}\in \Pi_1^m}\inf_{\pi_2\in\Pi_2^m}J(\pi_1,\pi_2,t,x)=\inf_{\pi_2\in\Pi_2^m}\sup_{\pi_{1}\in \Pi_1^m}J(\pi_1,\pi_2,t,x)$  for  $(t,x)\in [0,T]\times S$.
\item[(c)] The value function $\mathbb{M}(x)$ exists and equals to ${\varphi}(0,x)$, for any $x\in S$.
\end{description}

}

\end{prop}
The proof is rather long and is therefore presented in Appendix B.

Next, we use Proposition \ref{Th-4.2} to prove our main results by
approximation from  nonnegative payoff rates to the payoff rates
that may be unbounded from above and from below.

\begin{thm}\label{Th-4.3} {\rm Under Assumptions \ref{ass:3.1}, \ref{ass:3.1b} and \ref{ass:3.3}, the following assertions hold.
\begin{description}
\item[(a)] There exists a unique ${\varphi}$ in $\mathbb{B}^{1}_{V_0,V_1}([0,T]\times S)$  satisfying the following {\it Shapley equation} for the risk-sensitive criterion of CTSGs on the finite horizon:
\begin{eqnarray}
   \begin{cases}
    \displaystyle {\varphi}'(t,x)+\sup_{\phi\in P(A(x))}\inf_{\psi\in P(B(x))}\left[\theta r(x,\phi,\psi){\varphi}(t,x)+\int_{S}{\varphi}(t,y)q(dy|x,\phi,\psi)\right]=0, \\
    {\varphi}(T,x)= {e^{\theta g(x)}}, \label{EE-2}
    \end{cases}
 \end{eqnarray}
for each $x\in S$ and   $t\in E_{({\varphi},x)}$ with $m_L(E_{({\varphi},x)}^c)=0$.

\item[(b)] There exists a pair of Markov policies  $(\hat\pi_1,\hat\pi_2) \in \Pi_1^m\times \Pi_2^m$ such that, for $(t,x)\in [0,T]\times S$,
    \begin{eqnarray*}
     -{\varphi}'(t,x)&=&\theta r(x,\hat\pi^t_1,\hat\pi^t_2){\varphi}(t,x)+\int_{S}{\varphi}(t,y)q(dy|x,\hat\pi^t_1,\hat\pi^t_2) \\
                    &=& \sup_{\phi\in P(A(x))}\left[\theta r(x,\phi_,\hat\pi^t_2){\varphi}(t,x)+\int_{S}{\varphi}(t,y)q(dy|x,\phi,\hat\pi^t_2)\right] \\
                     &=& \inf_{\psi\in P(B(x))} \left[\theta r(x,\hat\pi^t_1,\psi){\varphi}(t,x)+\int_{S}{\varphi}(t,y)q(dy|x,\hat\pi^t_1,\psi)\right]
      \end{eqnarray*}
\item[(c)] $\displaystyle {\varphi}(t,x)= \sup_{\pi_{1}\in \Pi_1^m}\inf_{\pi_2\in\Pi_2^m}J(\pi_1,\pi_2,t,x)=\inf_{\pi_2\in\Pi_2^m}\sup_{\pi_{1}\in \Pi_1^m}J(\pi_1,\pi_2,t,x)=J(\hat\pi_1,\hat\pi_2,t,x)$  for  $(t,x)\in [0,T]\times S$. 
    \item[(d)] The value function $\mathbb{M}(x)$ exists and is equal to ${\varphi}(0,x)$ for each $x\in S$, and $(\hat\pi_1,\hat\pi_2)$ is a Nash equilibrium.
\end{description}}
\end{thm}

\begin{proof} We only prove (a) and the Nash Equilibrium policies $(\hat\pi_1,\hat\pi_2)$ since the others can be proved as  Proposition \ref{P-4.1}. For each $n\geq 1$, define $r_n$ on $K$ as follows:  for each $(x,a,b)\in K,$
 \begin{equation*}
r_{n}(x,a,b):=\max\{-n, r(x,a,b)\},\quad{ g_{n}(x):=\max\{-n, g(x)\}},
\end{equation*}
which implies that $\displaystyle\lim_{n\to\infty} r_n(x,a,b)=r(x,a,b)$, $\displaystyle\lim_{n\to\infty} g_n(x)=g(x)$ and $k_n(x,a,b):=r_n(x,a,b)+n\geq 0$ for each $(x,a,b)\in K$ and $f_{n}(x):=n+g_{n}(x)\ge 0 $ for
$(t,x)\in [0,T]\times S, n\geq 1$. Moreover, it follows from Assumption \ref{ass:3.1}(iii) that
\begin{equation*}
 -M_0-\frac{\sqrt{2}}{2}\sqrt{\ln V_0(x)}\leq \max\{-n, -M_0-\frac{\sqrt{2}}{2}\sqrt{\ln V_0(x)}\}\leq r_n(x,a,b)\leq M_0+\frac{\sqrt{2}}{2}\sqrt{\ln V_0(x)}.
\end{equation*}
Thus, $|r_n(x,a,b)|\leq M_0+\frac{\sqrt{2}}{2}\sqrt{\ln V_0(x)}$ for all $(x,a,b)\in K$ for all $n\geq 1$. {By
the same reasoning, we have $|g_n(x)|\leq M_0+\frac{\sqrt{2}}{2}\sqrt{\ln V_0(x)}$ for all $x\in S$ for all $n\geq 1$.} And so Assumptions \ref{ass:3.1}, \ref{ass:3.1b} and \ref{ass:3.3}  still hold for each model ${\cal{N}}_n$ defined by
\begin{equation}\label{MNn}
{\cal{N}}_n:=\left\{S,(A,A(x), x\in S), (B,B(x), x\in S), k_n(x,a,b), q(dy|x,a,b), {f_{n}(x)} \right\}.
\end{equation}

{
  For any real-valued Borel measurable function $u$ on $K$ and $g$ on $[0,T]\times S$ , let \begin{equation}J_{u,g}(t,x):=\sup_{\pi_{1}\in\Pi_{1}^{m}}\inf_{\pi_{2}\in\Pi_{2}^m}\mathbb{E}_\gamma^{\pi_{1},\pi_{2}}
  \left[e^{\theta\int_t^T\int_{A\times B}u(x_s,\pi_{1}(da|x_s,s)),\pi_{2}(db|x_s,s))ds+\theta g(x_{T})}|x_t=x\right],\label{PI}\end{equation} provided the integral exists.  Then, for each $n\geq 1$, since $k_n(x,a,b)\geq 0$ and $f_n(x)\geq 0$, by Proposition \ref{Th-4.2}(b) we have
  $J_{k_n,f_n}$ is in $\mathbb{B}_{V_0,V_1}^{1}([0,T]\times S)$ and satisfies
 \begin{eqnarray}
   \begin{cases}
  \displaystyle   J_{k_n,f_n}'(t,x)+\sup_{\phi\in P(A(x))}\inf_{\psi\in P(B(x))}\left[\theta k_n(x,\phi,\psi)J_{k_n,f_n}(t,x)+\int_{S}J_{k_n,f_n}(t,y)q(dy|x,\phi,\psi)\right]=0,\\
     J_{k_n,f_n}(T,x)=e^{\theta f_{n}(x)},
    \end{cases}  \label{g1}
   \end{eqnarray}
   for all $x\in S$ and $t\in E_{(J_{k_n,f_n},x)}$.

   Moreover,  since $J_{k_n,f_n}(t,x)=J_{r_n+n,g_{n}+n}(t,x)=J_{r_n,g_{n}}(t,x)e^{\theta(T-t)n+\theta n}$, by (\ref{g1}) we derive that
   \begin{eqnarray*}
   \begin{cases}
 \displaystyle    J_{r_n,g_{n}}'(t,x)+\sup_{\phi\in P(A(x))}\inf_{\psi\in P(B(x))}[\theta r_n(x,\phi,\psi)J_{r_n,g_{n}}(t,x)+\int_{S}J_{r_n,g_{n}}(t,y)q(dy|x,\phi,\psi)]=0,\\
     J_{r_n,g_{n}}(T,x)=e^{\theta g_{n}(x)}.
    \end{cases}
   \end{eqnarray*}
   This is
     \begin{eqnarray}
        J_{r_n,g_{n}}(t,x)&=&e^{\theta g_{n}(x)}+\int_t^T \sup_{\phi\in P(A(x))}\inf_{\psi\in P(B(x))}\left[\theta r_n(x,\phi,\psi)J_{r_n,g_{n}}(s,x)\right.\nonumber\\
        &&\left.+\int_{S}J_{r_n,g_{n}}(s,y)q(dy|x,\phi,\psi)\right]ds.\label{g1-a}
        \end{eqnarray}
 On the other hand, for each  $(t,x)\in [0,T]\times S$. it follows from  (\ref{PI}) and Lemma \ref{lem3.1}(c) that
 \begin{eqnarray}
 |J_{r_n,g_{n}}(t,x)|\leq LV_0(x), \ \ \ \  n\geq 1. \label{MN}
 \end{eqnarray}
     Since $r_n(x,a,b)$ and $g_{n}(x)$ are decreasing in $n\geq 1$, and so is the corresponding value functions $J_{r_n,g_{n}}(t,x)$. Therefore,   the limit  ${\varphi}(t,x):=\lim_{n\to\infty}J_{r_n,g_{n}}(t,x)$ exists for each $(t,x)\in [0,T]\times S$. Then, as the arguments for Proposition \ref{Th-4.2}  with ${\varphi}_n(t,x)$  replaced with  $J_{r_n,g_{n}}(t,x)$ here, from   (\ref{g1-a}) and (\ref{MN}) we can see that (a) is also true.

     Moreover, by (\ref{EE-2}) and part (c), 
     we see that $(\hat\pi_1,\hat\pi_2)$  is  a Nash equilibrium.
     }
  \end{proof}

\section{Algorithm}\label{sectional}

Until now, we have established the existence of the value function and Nash
equilibria of the risk-sensitive stochastic game. In this section,
under the suitable conditions, we prove that the value function and Nash
equilibria of the game can be approximated by iteratively solving a
series of two-person zero-sum matrix games through a value
iteration-type algorithm.


First, we have the following convergence result.

\begin{thm}\label{lem4.1}
{\rm  Under Assumption \ref{ass:3.3}, suppose that  $\|q\|=\sup_{x\in S} q^*(x)$ and $\|r\|$ { and $\|g\|$} are finite.
Let $v_{0}(t,x)$ be an arbitrary function in $\mathbb{B}^{1}_{1,1}([0,T]\times S)$, for $n\ge 0$,
\begin{eqnarray}
&&\hspace{-2cm} v_{n+1}(t,x):=\Gamma v_{n}(t,x)\nonumber\\
&&\hspace{-2cm}={e^{\theta g(x)}}+\int_t^T \sup_{\phi\in
P(A(x))}\inf_{\psi\in P(B(x))}\left[\theta
r(x,\phi,\psi)v_{n}(s,x)+\int_{S}v_{n}(s,y)q(dy|x,\phi,\psi)\right]ds,\label{nn+1}
\end{eqnarray}
where the operator $\Gamma$ is defined in Appendix (\ref{T21}).
 Then,  we have
\begin{description}
\item[(a)] $\lim_{n\to \infty}v_n(t,x)=\varphi(t,x)$, for any $(t,x)\in [0,T]\times S$,   and given $\varepsilon>0$, there exists $N_{1}>0$ such that, for all $n\geq N_{1}$,
\begin{eqnarray}\label{14}
\|v_{n+1}-v_{n}\|&=&\sup_{(t,x)\in [0,T]\times S}|v_{n+1}(t,x)-v_{n}(t,x)|<\frac{\varepsilon}{2e^{\left(\theta\|r\|+2\|q\|\right)T}\left(1+\frac{2\|q\|}{\theta \|r\|}\right)},\\
\|\varphi-v_{n+1}\|&=&\sup_{(t,x)\in [0,T]\times S}|\varphi(t,x)-v_{n+1}(t,x)|<\frac{\varepsilon}{2}, \label{n+1}
\end{eqnarray}
where $\varphi(t,x)$ is the value function of the stochastic game, same as in Proposition \ref{P-4.1}.
\item[(b)]There exist  Markov policies $\phi_{n} \in \Pi_1^m, \psi_{n} \in
\Pi_2^m$ such
that
\begin{eqnarray}\nonumber
&&(\phi_{n}(da|x,t),\psi_{n}(db|x,t))\in\\
&&\arg \sup_{\phi\in P(A(x))}\inf_{\psi\in P(B(x))}\left[\theta
r(x,\phi,\psi)v_n(t,x)+\int_{S}v_n(t,y)q(dy|x,\phi,\psi)\right],\label{15}
\end{eqnarray}
and  for each $n\geq N_{1}$,
\begin{eqnarray*}
\sup_{(t,x)\in [0,T]\times S}|J(\phi_{n},\psi_{n},t,x)-\varphi(t,x)|<\varepsilon,
\end{eqnarray*}
where $J(\phi_{n},\psi_{n},t,x)$ is defined in (\ref{eqtx}) and
$(\phi_{n},\psi_{n})$ is a strategy pair for $\varepsilon$-Nash
equilibrium.
\end{description}
}
\end{thm}
Since the operator $\Gamma$ is not a one-step contraction operator which
is widely used in discounted MDPs, the proof of Theorem~\ref{lem4.1}
is not trivial and is a little complicated. We present it in
Appendix C.

From Theorem~\ref{lem4.1}, we can see the calculation of function
$v_{n+1}(t,x)$ is critical for solving our risk-sensitive stochastic
game. Below, we describe the algorithm in a recurrent form. For the
current iteration $n$, we assume that we have the function
$v_n(t,x)$, where $(t,x) \in [0,T]\times S$. The purpose is to give
a computation procedure to iteratively obtain $v_{n+1}(t,x)$
converging to the value function of the game.


For any given
$(t,x) \in [0,T]\times S$, define
\begin{eqnarray}\label{eq_Cweight}
\hspace*{-0.5cm}&&c(t,x,v_n,a,b) := \theta
r(x,a,b)v_n(t,x)+\int_{S}v_n(t,y)q(dy|x,a,b), \ \forall
a\in A(x), b\in B(x).
\end{eqnarray}
Then,  the stochastic
game with the payoff function $c(t,x,v_n,a,b)$ can be treated as a \emph{matrix game} at the current
situation. We solve the corresponding two-person zero-sum matrix game through two
linear programs as below \cite{BarronBook}.
\begin{eqnarray}\label{eq_LPv}
&&\max \ v  \nonumber\\
&&\mbox{subject to: }\left\{
\begin{array}{l}
v \leq \int_{a\in A(x)} c(t,x,v_n,a,b) \phi(da), \quad \forall b \in B(x);\\
\int_{a\in A(x)} \phi(da) = 1; \\
\phi(da) \geq 0, \quad \forall a \in A(x).
\end{array}
\right.,
\end{eqnarray}
where $\phi$ is an optimization variable which
indicates the mixed policy (probability distribution) of player~1's action
selection in the action space $A(x)$. Similarly, we solve the
player~2's action selection probability distribution $ \psi$
through the following linear program.
\begin{eqnarray}\label{eq_LPz}
&&\min \ z  \nonumber\\
&&\mbox{subject to: }\left\{
\begin{array}{l}
z \geq \int_{b\in B(x)} c(t,x,v_n,a,b) {\psi}(db), \quad \forall a \in A(x);\\
\int_{b\in B(x)} {\psi}(db) = 1; \\
\tilde{\psi}(db) \geq 0, \quad \forall b \in B(x).
\end{array}
\right.
\end{eqnarray}
We can solve the above linear programs by simplex algorithms and
obtain the solutions $(\phi_{n}(da|x,t), a_{n}(t,x))$ and
$(\psi_{n}(db|x,t),b_{n}(t,x))$, respectively. With the classical
results of matrix games, we see that $a_{n}(t,x) =b_{n}(t,x)$, which
is the value of the matrix game. By (\ref{eq_Cweight}),  for any
$\phi\in P(A(x)), \psi\in P(B(x))$, we denote
 \begin{eqnarray*}
&&c(t,x,v_n,\phi,\psi) := \int_{a\in A(x)}\int_{b\in B(x)} \left[\theta
r(x,a,b)v_n(t,x)+\int_{S}v_n(t,y)q(dy|x,a,b)\right]\phi(da)\psi(db).
\end{eqnarray*}
Then, we have
\begin{eqnarray}
v_{n+1}(t,x)&=&e^{\theta g(x)}+
\int_t^T\sup_{\phi\in P(A(x))}\inf_{\psi\in P(B(x))}\left[\theta
r(x,\phi,\psi)v_n(s,x)+\int_{S}v_n(s,y)q(dy|x,\phi,\psi)\right]ds\nonumber\\
&=&e^{\theta g(x)}+
\int_t^T\sup_{\phi\in P(A(x))}\inf_{\psi\in P(B(x))}c(s,x,v_n,\phi,\psi) ds\nonumber\\
&=&e^{\theta g(x)}+
\int_{t}^{T}a_{n}(s,x)ds.\label{vn+1int}
\end{eqnarray}

Therefore, the value of $v_{n+1}(t,x)$ can be obtained by solving
(\ref{eq_LPv}) or (\ref{eq_LPz}). 
 We update the value of $v_{n+1}(t,x)$ for every $(t,x)$ by solving
the above linear programs (we only need to solve either
(\ref{eq_LPv}) or (\ref{eq_LPz}) since their optimal values are
equal). We repeat the same computation procedure based on the
updated values $v_{n+1}(t,x)$'s.

By Theorem \ref{lem4.1}, we have the following value iteration-type
Algorithm 1. For the convenience of computation, we assume that $A, B$
are finite.

%
%
%

\begin{algorithm}[!htb]\label{algo1}
    \caption{A value iteration-type algorithm to solve risk-sensitive stochastic games}
    {\bfseries Algorithm parameter:} the payoff rate $r(x,a,b)$, $\|r\|=\sup_{x\in S,a\in A(x),b\in B(x)}|r(x,a,b)|$; the transition rate $q(dy|x,a,b)$, $\|q\|=\sup_{x\in S} q^*(x)
    =\sup_{x\in S,a\in A(x),b\in B(x)}|-q(\{x\}|x,a,b)|$; the terminal reward $g(x)$; the risk-sensitive parameter $\theta>0$;  finite horizon $T>0$; Player 1 has $m=|A|$ actions and Player 2
    has $n=|B|$ actions; a small error bound $\varepsilon>0$ determining the algorithm accuracy
    \\
    {\bfseries Initialize:} $v_0(t,x)\in \mathbb{B}^{1}_{1,1}([0,T]\times S)$  arbitrarily, $n=0$\\
    \Repeat{ $\Delta<\frac{\varepsilon}{2e^{(\theta\|r\|+2\|q\|)T}\left(1+\frac{2\|q\|}{\theta \|r\|}\right)} $ }
    {
        $\Delta \gets 0$\\
        {\bfseries Loop} \For{ each $(t,x)\in [0,T]\times S$}
         {
         \For{ each $s\in [t,T]$}
        { 
            \For{ $i=1;i<m;i++$}
            {
                \For{$j=1;j<n;j++$}
                {
                    $c(s,x,v_{n})_{ij} \gets \theta
r(x,a_{i},b_{j})v_n(s,x)+\int_{S}v_n(s,y)q(dy|x,a_{i},b_{j})$
                }
            }
            Solving the game with matrix $C(s,x,v_{n})$\\
            $a_{n}(s,x) \gets  \sup_{\phi\in P(A(x))}\inf_{\psi\in P(B(x))}\phi^{T} C(s,x,v_{n})\psi $   \\
            $(\phi_{n}(\cdot|x,s),\psi_{n}(\cdot|x,s)) \gets \arg\sup_{\phi\in P(A(x))}\inf_{\psi\in P(B(x))}\phi^{T} C(s,x,v_{n})\psi$\\
       }
            $v_{n+1}(t,x)\gets e^{\theta g(x)}+\int_{t}^{T}a_{n}(s,x)ds$\\

            $\Delta \gets \max\{\Delta,|v_{n+1}(t,x)-v_{n}(t,x)|\}$\\

        }
        $n \gets n+1$\\
    }
    {\bfseries Output:}\\
   ~~~~~~~ $v_{n+1}(0,x)$ and $(\phi_{n}(\cdot|x,t),\psi_{n}(\cdot|x,t)),$ $(t,x)\in [0,T]\times S$\\
\end{algorithm}

With Theorem~\ref{lem4.1}, we can see
that the limit of $v_{n+1}(0,x)$ (as $n\to\infty$) is the value of
the game ${\varphi}(0,x)$. When the stopping condition (\ref{14}) is
satisfied and the algorithm stops, the output policy pair
$(\phi_{n},\psi_{n})$ is an $\varepsilon$-Nash equilibrium and the
output value function $v_{n+1}(0,x)$ is within $\varepsilon/2$ error bound from
the optimal value function ${\varphi}(0,x)$. The convergence of the
algorithm is also guaranteed, since it is proved that the algorithm
will stop within a finite number of iterations by
Theorem~\ref{lem4.1}(a).

If the payoff rates, the transition rates, or the terminal rewards is not bounded, we can solve  a series of the stochastic model
${\mathbb{M}}_n^+$  (\ref{control4}) or ${\cal{N}}_n$  (\ref{MNn}) by finite approximation technique.


\section{Examples}\label{section5}
In this section, we give two examples to illustrate our main
results.



 \begin{exm}\label{Example2}

In a system, the state of this system is  $x_{t}$ at time $t$ which is continuous in time.  The
corresponding state space is $S:=[0,+\infty)$. The system evolves as follows. If the system is at state $x_{t}$ at time $t$, two players play the game of scissors, paper, stone. $A=A(x)=B=B(x)=\{1,2,3\}$, $x\in S$, scissors are denoted as 1, paper is denoted as 2, and stone is denoted as 3.
Denote $a_{t}$ the action player 1 taken and $b_{t}$ the action player 2 taken at time $t$. 
The winner receives payoffs at rate $\alpha \sqrt{\ln (1+x_{t})}$
from the loser, $0<\alpha\le 0.5$. If they are tied, both of them
receive 0.
That is, $$r(x,a,b)=\left\{\begin{array}{ll}0,& x\in S, a=b;\\ \alpha \sqrt{\ln(1+x)},& x\ge 0,(a=1,b=2),(a=2,b=3),(a=3,b=1);\\
-\alpha \sqrt{\ln (1+x)},& x\ge 0,(a=2,b=1),(a=3,b=2),(a=1,b=3).
\end{array}\right.$$ Next, state $x_{t}$ is assumed to keep invariable for an exponential-distributed random time with parameter $\lambda(x_{t},a_{t},b_{t})$ ($0<\lambda(x,a,b)\le L$), and then jump to other states with exponential-distribution $\rm{exp}(\frac{1}{x_{t}})$.
Therefore, the transition rate of state is represented by, for each
$D\in \mathcal{B}(S)$,
\begin{equation*}
q(D|x,a,b)=\lambda(x,a,b)\left[\int_{y\in D}\frac{1}{x}e^{
-\frac{y}{x}}dy-\delta_{x}(D)\right], \quad x\in S, a\in A, b\in
B.\end{equation*} The terminal reward function is
$g(x)=\frac{\sqrt{\ln (1+x)}}{2}$. For this zero-sum stochastic game
model, player~1 wishes to maximize the risk-sensitive rewards on a
given $T$ horizon over all policies and player~2 wishes to minimize
the risk-sensitive cost on a given $T$ horizon over all policies.
\end{exm}

Under the above conditions, we have the following fact.
\begin{prop}\label{prop6.2}{\rm
 Example \ref{Example2} satisfies  Assumptions \ref{ass:3.1}, \ref{ass:3.1b} and \ref{ass:3.3}, and hence (by Theorem \ref{Th-4.3}) there exists a Nash equilibrium.}
\end{prop}

\begin{proof}
Let us first show that Assumption \ref{ass:3.1} holds.

There exist a real-valued Borel measurable function $V_0(x):=1+x\geq 1$ on $S$ and  positive constants
$\rho_0=1,  L_0=L, M_0=1$, such that
\begin{description}
\item[(i)]for any $x\in S, a\in A, b\in B$, \begin{eqnarray*}
 &&\int_{S}V_0(y)q(dy|x,a,b)\\&=&\lambda(x,a,b)\left[\int_{0}^{+\infty}
(1+y)\frac{1}{x}e^{-\frac{y}{x}}dy-(1+x)\right]\\
&=&\lambda(x,a,b)\left[ 1+x-(1+x)\right]=0\le 1+x=\rho_{0}V_{0}(x);
 \end{eqnarray*}

\item[(ii)] $\displaystyle q^*(x)=\sup_{a\in A(x),b\in B(x)}q(x,a,b)=\sup_{a\in A(x),b\in B(x)}-q(\{x\}|x,a,b)=
 \sup_{a\in A(x),b\in B(x)}\lambda(x,a,b)\le L\leq L(1+x)$ for all $x\in S$, where $q^*(x)$ is as in (\ref{Q});

\item[(iii)]  $|r(x,a,b)|\le \alpha \sqrt{\ln (1+x)} \le 0.5\sqrt{\ln (1+x)}\leq 1+\frac{\sqrt{2}}{2}\sqrt{\ln V_0(x)}$ and $|g(x)|=\frac{\sqrt{\ln (1+x)}}{2}\leq 1+ \frac{\sqrt{2}}{2}\sqrt{\ln V_0(x)} $
for any $x\in S, a\in A, b\in B$.

\end{description}

Now we show that Assumption \ref{ass:3.1b} holds.
A directive calculation gives
$\int_{0}^{\infty}y^{k}\frac{1}{x}e^{-\frac{y}{x}}dy
=k!x^{k}$ for all $k=0,1,\ldots$.

 There exist a real-valued Borel measurable function  $V_1(x):=(1+x)^2  \geq 1$ on $S$, and positive  constants
$\rho_1=23L, b_1=1$,  and  $M_1=1$,  such that
\begin{description}
\item[(i)] for any $x\in S, a\in A, b\in B$,
 \begin{eqnarray*}
 &&\int_{S}V_1^2(y)q(dy|x,a,b)\\&=&\lambda(x,a,b)\left[\int_{0}^{+\infty}
(1+y)^4\frac{1}{x}e^{-\frac{y}{x}}dy-(1+x)^4\right]\\
&=&\lambda(x,a,b)\left[1+4x+12x^2+24x^3+24x^4-1-4x-6x^2-4x^3-x^4\right]\\
&=&\lambda(x,a,b)\left[6x^2+20x^3+23x^4\right]\le 23L(1+x)^4=\rho_{1}V_{1}^2(x)+b_{1};
 \end{eqnarray*}

\item[(ii)] $V_0^2(x)=(1+x)^2=M_{1}V_1(x)=(1+x)^2$ for all $x\in S$.
\end{description}
Since $A, B$ are finite, Assumption  \ref{ass:3.3}  holds.

Thus, Assumptions \ref{ass:3.1}, \ref{ass:3.1b} and \ref{ass:3.3}
hold for Example \ref{Example2}, and then Theorem \ref{Th-4.3}
guarantees the existence of a Nash equilibrium.

\end{proof}

\begin{exm}\label{Example3}
\begin{equation*}{\mathbb{M}}_2:=\big\{S,(A,A(x), x\in S), (B,B(x), x\in S),r(x,a,b), q(dy|x,a,b), g(x) \big\},\end{equation*}
where $S=(-\infty,\infty)$,
for each $D\in \mathcal{B}(S)$,
\begin{equation}\label{q}
q(D|x,a,b)=\lambda(x,a,b)\left[\int_{y\in
D}\frac{1}{\sqrt{2\pi}\sigma}e^{
-\frac{(y-x)^2}{2\sigma^2}}dy-\delta_{x}(D)\right], \ \ x\in S, a\in
A(x), b\in B(x).\end{equation}
\end{exm}

To ensure the existence of a Nash equilibrium for the
model, we consider the following hypotheses:
\begin{itemize}
\item[$(A_1)$] $0<\lambda(x,a,b)\leq M(x^2+1)$ , $|r(x,a,b)|\leq M_0+\frac{\sqrt{2}}{2}\sqrt{\ln (1+x^2)}$ for all $x\in S, a\in A(x), b\in B(x)$, $|g(x)|\leq M_0+\frac{\sqrt{2}}{2}\sqrt{\ln (1+x^2)}$ for all $x\in S$ with some positive constants $M$ and $M_{0}$;
\item[$(A_2)$]  $A(x),B(x)$ are assumed to be a compact set of  Borel spaces $A,B$ for each $x\in S$, respectively;
\item[$(A_3)$] $\lambda(x,a,b)$ and $r(x,a,b)$  are Borel measurable on $K$ and  continuous in $a\in A(x), \ b\in B(x)$ for each fixed $x\in S$.
\end{itemize}

Under the above conditions, we have the following fact.
\begin{prop}\label{prop6.3}{\rm
 Example \ref{Example3} satisfies  Assumptions \ref{ass:3.1}, \ref{ass:3.1b} and \ref{ass:3.3}, and hence (by Theorem \ref{Th-4.3}) there exists a Nash equilibrium.}
\end{prop}

\begin{proof}
Let us first show that Assumption \ref{ass:3.1} holds.

There exist a real-valued Borel measurable function $V_0(x):=1+x^2\geq 1$ on $S$ and  positive constants
$\rho_0=M\sigma^2,  L_0=M, M_0=M_{0}$, such that
\begin{description}
\item[(i)]for any $x\in S, a\in A(x), b\in B(x)$,
 \begin{eqnarray*}
 \int_{S}V_0(y)q(dy|x,a,b)&=&\lambda(x,a,b)\left[\frac{1}{\sqrt{2\pi}\sigma}\int_{-\infty}^{+\infty}
(y^2+1)e^{-\frac{(y-x)^2}{2\sigma^2}}dy-(x^2+1)\right]\\
&=&\lambda(x,a,b) \sigma^2\leq M\sigma^2V_0(x); \end{eqnarray*}

\item[(ii)] $\displaystyle q^*(x)=\sup_{a\in A(x),b\in B(x)}q(x,a,b)=\sup_{a\in A(x),b\in B(x)}-q(\{x\}|x,a,b)=
 \sup_{a\in A(x),b\in B(x)}\lambda(x,a,b)\le  M(x^2+1)=M V_{0}(x)$ for all $x\in S$, where $q^*(x)$ is as in (\ref{Q});

\item[(iii)]  $|r(x,a,b)|\le     M_0+\frac{\sqrt{2}}{2}\sqrt{\ln (1+x^2)}=    M_0+\frac{\sqrt{2}}{2}\sqrt{\ln V_{0}(x)}$,  $x\in S, a\in A(x), b\in B(x)$, and $|g(x)|\leq     M_0+\frac{\sqrt{2}}{2}\sqrt{\ln (1+x^2)}=  M_0+\frac{\sqrt{2}}{2}\sqrt{\ln V_{0}(x)} $
for any $x\in S$.

\end{description}

Now we show that Assumption \ref{ass:3.1b} holds.

A directive calculation gives $\frac{1}{\sqrt{2\pi}\sigma}\int_{-\infty}^{\infty} (y-x)^{2k+1}e^{-\frac{(y-x)^2}{2\sigma^2}}dy=0$ and $\frac{1}{\sqrt{2\pi}\sigma}\int_{-\infty}^{\infty}(y-x)^{2k}e^{-\frac{(y-x)^2}{2\sigma^2}}dy =1\cdot 3\cdots (2k-1)\sigma^{2k}$ for all $k=0,1,\ldots$.

 There exist a real-valued Borel measurable function  $V_1(x):=1+x^4  \geq 1$ on $S$, and positive  constants
$\rho_1=3780M\left(\sigma^8+\sigma^6+\sigma^4+\sigma^2\right), b_1=1$,  and  $M_1=2$,  such that
\begin{description}
\item[(i)] for any $x\in S, a\in A(x), b\in B(x)$,
\begin{eqnarray*}
&&\int_{S}V_1^2(y)q(dy|x,a,b)\\&=&\lambda(x,a,b)\left[\frac{1}{\sqrt{2\pi}\sigma}\int_{-\infty}^{+\infty}
(y^4+1)^2e^{-\frac{(y-x)^2}{2\sigma^2}}dy-(x^4+1)^2\right]\\
&=&\lambda(x,a,b)\left(105\sigma^8+420x^2\sigma^6+210x^4\sigma^4+6\sigma^4+12\sigma^2x^2+28x^6\sigma^2\right) \\
&\leq&420 \lambda(x,a,b)\left(\sigma^8+\sigma^6+\sigma^4+\sigma^2\right)\left(x^6+x^4+x^2+1\right)\\
&\leq&420 \lambda(x,a,b)\left(\sigma^8+\sigma^6+\sigma^4+\sigma^2\right)\left(3x^6+3\right)\\
&\leq&1260 M\left(\sigma^8+\sigma^6+\sigma^4+\sigma^2\right)\left(x^6+1\right)(1+x^2)\\
&\leq&3780M\left(\sigma^8+\sigma^6+\sigma^4+\sigma^2\right)\left(x^4+1\right)^2\\
&\le &3780M\left(\sigma^8+\sigma^6+\sigma^4+\sigma^2\right)V_1^2(x)+1.
 \end{eqnarray*}

\item[(ii)] $V_0^2(x)=(1+x^2)^2\le 2(1+x^4)$ for all $x\in S$.
\end{description}
%
%


Thus,  Assumptions \ref{ass:3.1}, \ref{ass:3.1b} and \ref{ass:3.3} (under the hypotheses $A_1$--$A_3$) hold for Example \ref{Example3}, and then Theorem \ref{Th-4.3} gives the existence of a Nash equilibrium.
\end{proof}

\begin{rem}\label{rem3.8}
{\rm In Example \ref{Example2}, the payoff rates $r(x,a,b)$ are
allowed to be unbounded from above. In Example \ref{Example3}, the
payoff rates $r(x,a,b)$, the transition rates $q(dy|x,a,b)$ and the
terminal reward $g(x)$ are all unbounded from below and from above.

}
\end{rem}

\section{Conclusion}\label{section6}

In this paper we have studied a finite-horizon two-person zero-sum
risk-sensitive stochastic game for continuous-time Markov chains
with Borel state and action spaces, in  which the
payoff rates, the transition rates and the terminal rewards are allowed to be unbounded from below and from above
and the policies can be history-dependent. This model is a
generalization of that in  the existing literature \cite{BU-17} with
bounded payoff rates and Markov policies. To establish the
corresponding Shapley equation and the existence of a Nash
equilibrium for the general model, we develop a finite-approximation
technique. More specifically, for the bounded case (i.e., the
payoff rates, the transition rates and the terminal rewards are bounded), we first prove the
existence of a solution to the Shapley equation by the
Banach-fixed-point theorem with a $k$-step contraction operator,
establish the existence of both the value function  and a Nash equilibrium
for the stochastic game, and verify that the value function of the game
uniquely solves the Shapley equation by the extension of the
Dynkin's formula. Then, by developing a finite-approximation
technique, we extend the results for the bounded case to the general
case that the  payoff rates, the transition rates and the terminal rewards are unbounded (Theorem
\ref{Th-4.3}). As a consequence, our results extend the findings in
\cite{Wei18} and answer an open question posed there.

The computation of Nash equilibria is of significance and desirable
for the practical application of game theory. To the best of our
knowledge, our iteration algorithm developed in this paper for
computing the value function and Nash equilibria of risk-sensitive stochastic
games is a first attempt. We also prove the convergence of the
algorithm by a specific contraction operator. The combination of the
iteration algorithm with other approximation techniques  to handle
the issue of large scalability, such as reinforcement learning,
deserves further investigation in a regime of so-called multi-agent
reinforcement learning [10, 37].

\begin{appendices}

\section*{Appendix:}
\section{Proof of Proposition \ref{P-4.1}}
\begin{proof} (a)  Define the following operator $\Gamma$ on $\mathbb{B}_{1,1}^{1}([0,T]\times S)$ by
\begin{equation}
\Gamma{\varphi}(t,x):={e^{\theta g(x)}}+\int_t^T \sup_{\phi\in
P(A(x))}\inf_{\psi\in P(B(x))}\left[\theta
r(x,\phi,\psi){\varphi}(s,x)+\int_{S}{\varphi}(s,y)q(dy|x,\phi,\psi)\right]ds
\label{T21}
\end{equation}
for any $(t,x)\in [0,T]\times S$ and ${\varphi} \in
\mathbb{B}_{1,1}^{1}([0,T]\times S).$

Then, for each $(t,x)\in [0,T]\times S$, and any
$\varphi_{1},\varphi_{2}\in \mathbb{B}_{1,1}^{1}([0,T]\times S)$,
from (\ref{T21}) and $q(\{x\}|x,a,b)+q(S\setminus
\{x\})|x,a,b)\equiv 0$, we obtain
\begin{eqnarray*}
 |\Gamma\varphi_{1}(t,x)-\Gamma\varphi_{2}(t,x)|   &\leq & (\theta\|r\|+2\|q\|) \int_{t}^{T}\|\varphi_{1}-\varphi_{2}\|ds \nonumber \\
  &=&\tilde{L}(T-t)\|\varphi_{1}-\varphi_{2}\|, \label{Bf0}
\end{eqnarray*}
where $\tilde{L}:=\theta\|r\|+2\|q\|<\infty$. Furthermore,   by
induction we can  prove the following fact:
\begin{eqnarray}
|\Gamma^n\varphi_{1}(t,x)-\Gamma^n\varphi_{2}(t,x)| \leq \tilde{L}^n
\frac{(T-t)^n}{n!}\|\varphi_{1}-\varphi_{2}\|, \ \ \  \forall \
(t,x)\in [0,T]\times S,  n\geq 1. \label{Bf}
\end{eqnarray}
Since $\sum_{n=0}^\infty \tilde{L}^n
\frac{T^n}{n!}\|\varphi_{1}-\varphi_{2}\|=e^{\tilde{L}T}\|\varphi_{1}-\varphi_{2}\|<\infty$,
there exists some integer $k$ such that  the constant \begin{equation*}\label{beta}\beta:=
\tilde{L}^k \frac{T^k}{k!}<1.\end{equation*} Thus, by (\ref{Bf}) we have
$\|\Gamma^k\varphi_1-\Gamma^k\varphi_2\|\leq \beta
\|\varphi_{1}-\varphi_{2}\|$. Therefore, $\Gamma$ is a $k$-step contract
operator. Thus, there exists a function ${\varphi} \in
\mathbb{B}_{1,1}^{1}([0,T]\times S)$  such that $\Gamma{\varphi}={\varphi}$,
that is, for any $(t,x)\in[0,T]\times S,$
\begin{eqnarray}
 {\varphi}(t,x)={e^{\theta g(x)}}+\int_t^T \sup_{\phi\in P(A(x))}\inf_{\psi\in P(B(x))}\left[\theta r(x,\phi,\psi){\varphi}(s,x)+\int_S{\varphi}(s,y)q(dy|x,\phi,\psi)\right]ds. \label{C}
\end{eqnarray}
Since $\|q\|$, $\|r\|$ and $\|g\|$ are finite, by (\ref{C}) we see that ${\varphi}\in \mathbb{B}^{1}_{1,1}([0,T]\times S)$, and thus  (a) follows.

(b) By (a) and Fan's minimax theorem in \cite{Fan}, the minimax measurable selection theorems Theorem 2.2 in \cite{NA2} together with Lemma 4.1 in \cite{GuoZhang20},
we see that (b) is true.

(c)  Given any $\pi_1\in \Pi_1$, for each $x\in S$ and  $t\in E_{(\varphi,x)}$, by (\ref{S-2}) we have

\begin{eqnarray*}
   \begin{cases}
     -{\varphi}'(t,x)\geq \theta {\varphi}(t,x) \int_{A(x)}r(x,a,\hat\pi_2^t)\pi_1(da|\omega,t)+\int_{S}\int_{A(x)}{\varphi}(t,y)q(dy|x,a,\hat\pi_2^t)\pi_1(da|\omega,t),\\
     {\varphi}(T,x)=e^{\theta g(x)}.
    \end{cases}
\end{eqnarray*}
which, together with the fact that $x_t(\omega)$ is a piece-wise constant (by (\ref{T_k}), implies
\begin{eqnarray*}
 && -\left( e^{\int_0^t\int_{A(x_{v})}\theta r(x_v,a,\hat\pi_2^v)\pi_1(da|\omega,v)dv}{\varphi}(t,x_t) \right)'\\
&=& - e^{\int_0^t\int_{A(x_{v})}\theta r(x_v,a,\hat\pi_2^v)\pi_1(da|\omega,v)dv}\left\{\left[
\int_{A(x_{t})}\theta r(x_t,a,\hat\pi_2^t)\pi_1(da|\omega,t)\right]{\varphi}(t,x_t)+{\varphi}'(t,x_t) \right\}\\
&\ge & - e^{\int_0^t\int_{A(x_{v})}\theta r(x_v,a,\hat\pi_2^v)\pi_1(da|\omega,v)dv}\left[
\int_{A(x_{t})}\theta r(x_t,a,\hat\pi_2^t)\pi_1(da|\omega,t)\right]{\varphi}(t,x_t)\\
&&+e^{\int_0^t\int_{A(x_{v})}\theta r(x_v,a,\hat\pi_2^v)\pi_1(da|\omega,v)dv}\\
&&\left[\theta {\varphi}(t,x_{t}) \int_{A(x_{t})}r(x_{t},a,\hat\pi_2^t)\pi_1(da|\omega,t)+\int_{S}\int_{A(x_{t})}{\varphi}(t,y)q(dy|x_{t},a,\hat\pi_2^t)\pi_1(da|\omega,t)\right] \\
    &= &\int_{S}\int_{A(x_t)}q(dy|x_t,a,\hat\pi_2^t)\pi_1(da|\omega,t)\left(e^{\int_0^t\int_{A(x_{v})}\theta r(x_v,a,\hat\pi_2^v)\pi_1(da|\omega,v)dv}{\varphi}(t,y)\right), \ \ \ \ \forall \ t\ge 0.
\end{eqnarray*}

 Thus, by Lemma \ref{Th-3.2}(a) we have
 \begin{eqnarray*}
  &&\mathbb{E}_{x}^{\pi_1,\hat\pi_2} \left(e^{\int_0^T\int_{A(x_t)}\theta r(x_t,a,\hat\pi_2^t)\pi_1(da|\omega,t)dt+\theta g(x_{T})}\right)-{\varphi}(0,x) \\
  &=&\mathbb{E}_{x}^{\pi_1,\hat\pi_2} \left(e^{\int_0^T\int_{A(x_t)}\theta r(x_t,a,\hat\pi_2^t)\pi_1(da|\omega,t)dt}{\varphi}(T,x_T)\right)-{\varphi}(0,x)\leq 0
\end{eqnarray*}
and so
\begin{eqnarray}
J(\pi_1,\hat\pi_2,0,x)\leq {\varphi}(0,x) \ \ \ \ {\rm for
\  all} \ x\in S,\pi_1\in\Pi_1. \label{K1}
\end{eqnarray}
Moreover, by (\ref{eqtx}) a similar proof gives
\begin{eqnarray}
J(\pi_1,\hat\pi_2,t,x)\leq { \varphi}(t,x) \ \ \ \ {\rm for
\  all} \ t\in [0,T],x\in S,\pi_1\in\Pi_1^m. \label{K}
\end{eqnarray}
Therefore, since $\pi_1$ can be arbitrary, by (\ref{K1})-(\ref{K}) we have
\begin{eqnarray}
   \inf_{\pi_2\in \Pi_2}\sup_{\pi_1\in\Pi_1} J(\pi_1,\pi_2,0,x)\leq {\varphi}(0,x), \ {\rm  and} \   \inf_{\pi_2\in \Pi_2^{m}} \sup_{\pi_1\in\Pi_1^m}J(\pi_1,\pi_2,t,x)\leq {\varphi}(t,x), \label{U}
   \end{eqnarray}
  for   all $ (t,x)\in [0,T]\times S$.

  Furthermore, by (\ref{S-3}),   we have
\begin{eqnarray}
   &&\sup_{\pi_1\in \Pi_1}\inf_{\pi_2\in\Pi_2} J(\pi_1,\pi_2,0,x)\geq { \varphi}(0,x), \ {\rm  and} \   \ \sup_{\pi_1\in\Pi_1^m} \inf_{\pi_2\in \Pi_2^{m}} J(\pi_1,\pi_2,t,x)\geq {\varphi}(t,x),\ \ \ \ \ \ \ \ \ \\
  {\rm and} &&   J(\hat\pi_1,\hat\pi_2,0,x)={\varphi}(0,x), \quad \forall \ x\in S. \label{L}
   \end{eqnarray}
 By (\ref{U})-(\ref{L}) we have
  \begin{eqnarray}
   &&\sup_{\pi_1\in \Pi_1}\inf_{\pi_2\in\Pi_2} J(\pi_1,\pi_2,0,x)=\inf_{\pi_2\in\Pi_2}\sup_{\pi_1\in \Pi_1} J(\pi_1,\pi_2,0,x)={\varphi}(0,x), \label{LL}\\
  && \sup_{\pi_1\in \Pi_1^m}\inf_{\pi_2\in\Pi_2^m} J(\pi_1,\pi_2,t,x)=\inf_{\pi_2\in\Pi_2^m}\sup_{\pi_1\in \Pi_1^m} J(\pi_1,\pi_2,t,x)={\varphi}(t,x), \nonumber\\
    {\rm and} &&   J(\hat\pi_1,\hat\pi_2,t,x)={\varphi}(t,x), \ \ \ \ \forall \ (t,x)\in [0,T]\times S, \nonumber
   \end{eqnarray}
 which gives (c).

 (d) Obviously, (d) is from (\ref{LL}).

%

(e)  Fix  any $s,t\in [0,T]$ with $s<t$. Then,  for any  Markov policy $\pi_{1}\in \Pi_{1}^{M}$,
we define the corresponding Markov policy $\pi_{1,s}^t$ as follows: for each $x\in S$,
\begin{eqnarray}
&& \pi_{1,s}^{t}(x,v) = \begin{cases}
          \pi_{1}(x,v+t-s), & \mbox{$ v\geq s, \label{E-12}$} \\
   \pi_{1}(x,v), & \mbox{otherwise,}
    \end{cases}
\end{eqnarray}
where $\pi_{1}(x,t):=\pi_{1}(da|x,t)$ is a stochastic kernel on $A(x)$.
It is the similar notation for $\pi_{2}\in \Pi_{2}^{M}$. Then, we have,  for each $(x,v)\in  S\times [s, s+T-t]$,
$$q(dy|x,\pi_{1,s}^{t}(x,v),\pi_{2,s}^{t}(x,v))=q(dy|x,\pi_{1}(x,v+t-s),\pi_{2}(x,v+t-s)),$$  $$r(x,\pi_{1,s}^{t}(x,v),\pi_{2,s}^{t}(x,v))=r(x,\pi_{1}(x,v+t-s),\pi_{2}(x,v+t-s)).$$
Let
 \begin{eqnarray}
  &&  J(\pi_{1},\pi_{2},s\sim t,x):=\mathbb{E}_\gamma^{\pi_{1},\pi_{2}}
  \left[e^{\int_s^t\int_A\int_B\theta r(x_v,a,b)\pi_{1}(da|x_v,v)\pi_{2}(db|x_v,v)dv+\theta g(t,x_{t})}|x_s=x\right], \nonumber\\
   && J_*(s\sim t,x):=\sup_{\pi_{1}\in \Pi_{1}^{m}}\inf_{\pi_{2}\in \Pi_{2}^{m}}J(\pi_{1},\pi_{2},s\sim t,x). \label{E8}
\end{eqnarray}
By the Markov property of $\{x_t,t\geq 0\}$  (under any Markov policy $(\pi_{1},\pi_{2})$) and (\ref{E-12})-(\ref{E8}), { we have
$X_{u}$ under policies $\pi_{1},\pi_{2}$ and $X_{t}=x$ has the same distribution with $X_{u+s-t}$ under policies $\pi_{1,s}^{t},\pi_{2,s}^{t}$ and $X_{s}=x$ for
any $t\le u\le T$. Therefore,}
   $J(\pi_{1},\pi_{2},t\sim T,x)= J(\pi_{1,s}^{t},\pi_{2,s}^{t},s\sim T+s-t,x)$.
    From this, we have
    \begin{eqnarray*}&&\inf_{\pi_{2}\in \Pi_{2}^{M}}J(\pi_{1},\pi_{2},t\sim T,x)= \inf_{\pi_{2}\in \Pi_{2}^{M}}J(\pi_{1,s}^{t},\pi_{2,s}^{t},s\sim T+s-t,x)\\
    &\ge& \inf_{\pi_{2,s}^{t}\in \Pi_{2}^{M}}J(\pi_{1,s}^{t},\pi_{2,s}^{t},s\sim T+s-t,x),\quad \forall \pi_{1}\in \Pi_{1}^{M}.\end{eqnarray*}
    Similarly, we get
    \begin{eqnarray*}&&\inf_{\pi_{2,s}^{t}\in \Pi_{2}^{M}}J(\pi_{1,s}^{t},\pi_{2,s}^{t},s\sim T+s-t,x)=\inf_{\pi_{2,s}^{t}\in \Pi_{2}^{M}}J(\pi_{1},\pi_{2},t\sim T,x)\\
    &\ge& \inf_{\pi_{2}\in \Pi_{2}^{M}}J(\pi_{1},\pi_{2},t\sim T,x),\quad \forall \pi_{1}\in \Pi_{1}^{M}.\end{eqnarray*}
    Thus, $$\inf_{\pi_{2,s}^{t}\in \Pi_{2}^{M}}J(\pi_{1,s}^{t},\pi_{2,s}^{t},s\sim T+s-t,x)=\inf_{\pi_{2}\in \Pi_{2}^{M}}J(\pi_{1},\pi_{2},t\sim T,x),\quad \forall \pi_{1}\in \Pi_{1}^{M}.$$
    Similarly, we have $J_*(t\sim T,x)=J_*(s\sim T+s-t,x)$. Moreover, since $r(x,a,b)\geq 0$ on $K$, by (\ref{E8}) and $t>s$, we have $J_*(t\sim T,x)=J_*(s\sim T+s-t,x)\leq J_*(s\sim T,x)$, which, together with $J_*(t\sim T,x)={\varphi}(t,x)$ in part (c), gives  (e).
\end{proof}
\section{Proof of Proposition \ref{Th-4.2} }
\begin{proof}  We only prove part (a). This is because  parts (b) and (c)  can be proved as the same arguments of (c) and (d) of Proposition \ref{P-4.1}. First, under Assumption \ref{ass:3.1} (iii), we have
$$ 0\leq r(x,a,b)\leq \frac{\sqrt{2}}{2}\sqrt{\ln V_0(x)}+M_0 \ {\rm for \ all \ } (x,a,b)\in K, $$
and
$$ 0\leq g(x)\leq \frac{\sqrt{2}}{2}\sqrt{\ln V_0(x)}+M_0 \ {\rm for \ all \ } x\in S.$$
For each $n\geq 1$,    let $A_n(x):=A(x)$ and $B_n(x):=B(x)$ for $x\in S$,   $K_{n}:=\{(x,a,b)|x\in S, a\in
A_n(x),b\in B_n(x)\}$, and $S_n:=\{x\in S| V_0(x)\leq n\}$. Moreover, for each $x\in S$, $a\in A_n(x)$, $b\in B_n(x)$, let
\begin{equation}\label{app1}
q_{n}(dy|x,a,b):=
\begin{cases}
q(dy|x,a,b), &\text{if } x\in S_n,  \\
0, & \text{if } x\not\in S_n;
\end{cases}
\end{equation}
  \begin{equation}\label{app1a}
r_{n}^+(x,a,b):=
\begin{cases}
 \min\{n, r(x,a,b)\}, &\text{if } x\in S_n,\\
0, & \text{if } x\not\in S_n;
\end{cases}
\end{equation}
and
{
 \begin{equation}\label{app1ag}
g_{n}^+(x):=
\begin{cases}
 \min\{n, g(x)\}, &\text{if } x\in S_n,\\
0, & \text{if } x\not\in S_n.
\end{cases}
\end{equation}}
 Fix any $n\geq 1$. By (\ref{app1}),
it is obvious that $q_{n}(dy|x,a,b)$ denotes indeed
transition rates on $S$, which  are \emph{conservative} and \emph{stable}. Then,  we obtain a sequence of models
$\{\mathbb{M}_n^+\}$:
\begin{equation}\label{control4}
{\mathbb{M}}_n^+:=\big\{S,(A,A(x), x\in S), (B,B(x), x\in S),r_n^+(x,a,b), q_{n}(dy|x,a,b), g_{n}^{+}(x) \big\},
\end{equation}
for which the  payoff rates $r_n^+(x,a,b)$, the transition rates $q_n(dy|x,a,b)$  and terminal reward $g_{n}^{+}(x)$ are all bounded by Assumption \ref{ass:3.1}  and (\ref{app1})-(\ref{app1ag}). In the following arguments, any quality  with respect to ${\mathbb{M}}_n^+$ is labeled by a lower $n $, such  the risk-sensitive value  $J_n(\pi_1,\pi_2,t,x)$ of a pair of Markov policies $(\pi_1,\pi_2)$ and the value function $J_n(t,x):=\sup_{\pi_{1}\in \Pi_1^m}\inf_{\pi_2\in\Pi_2^m}J_{n}(\pi_1,\pi_2,t,x)$.

Obviously, Assumptions \ref{ass:3.1}, \ref{ass:3.1b} and
\ref{ass:3.3} still hold for    each model $\mathbb{M}_{n}^+$. Thus,
for each $n\geq 1$, it follows from Proposition \ref{P-4.1} that
there exists ${\varphi}_n(t,x)\in \mathbb{B}_{1,1}^{1}([0,T]\times S)$
satisfying   (\ref{EEE-2}) for the corresponding
 $\mathbb{M}_{n}^+$, that is,
 \begin{eqnarray}
   \begin{cases}
     \displaystyle {\varphi}_n'(t,x)+\sup_{\phi\in P(A(x))}\inf_{\psi\in P(B(x))}\left[\theta r_{n}^{+}(x,\phi,\psi){\varphi}_{n}(t,x)+\int_{S}{\varphi}_{n}(t,y)q_{n}(dy|x,\phi,\psi)\right]=0,&  \\
     {{\varphi}_n(T,x)=e^{\theta g_{n}^{+}(x)}}, \label{g1a}
    \end{cases}
   \end{eqnarray}
for each $x\in S$ and $t\in E_{({\varphi}_n,x)}$ with
$m_L(E_{({\varphi}_n,x)}^c)=0$.

From (\ref{S-1}) in Proposition \ref{P-4.1}(b), (\ref{g1a}) gives  the existence of
Markov policies $\phi_{n}(da|x,t)\in \Pi_1^m, \psi_{n}(db|x,t) \in
\Pi_2^m$ such that,
\begin{eqnarray}
   \begin{cases}
     {\varphi}_n'(t,x)+\theta r_n^+(x,\phi_{n}^{t},\psi_{n}^{t}){\varphi}_n(t,x)+\int_{S}{\varphi}_n(t,y)q_{n}(dy|x,\phi_{n}^{t},\psi_{n}^{t})=0,    \\
     {\varphi}_n(T,x)=e^{\theta g_{n}^{+}(x)}, \label{g2}
    \end{cases}
\end{eqnarray}
 for each $x\in S$ and $t\in E_{({\varphi}_n,x)}$ with $m_L(E_{({\varphi}_n,x)}^c)=0$.
{
From (\ref{app1}), (\ref{app1a}) and (\ref{app1ag}), we obtain
\begin{eqnarray*}
   \begin{cases}
     {\varphi}_n'(t,x)+\theta r_n^+(x,\phi_{n}^{t},\psi_{n}^{t}){\varphi}_n(t,x)+\int_{S}{\varphi}_n(t,y)q_{n}(dy|x,\phi_{n}^{t},\psi_{n}^{t})=0, x\in S_{n},   \\
     {\varphi}_n'(t,x)=0, x\notin S_{n},   \\
     {\varphi}_n(T,x)=e^{\theta g_{n}^{+}(x)}, x\in S_{n}, \\
     {\varphi}_n(T,x)=1,x\notin S_{n},
    \end{cases}
\end{eqnarray*}
 for each $t\in E_{({\varphi}_n,x)}$ with $m_L(E_{({\varphi}_n,x)}^c)=0$.

 From Proposition \ref{P-4.1}(b),
 \begin{eqnarray}
   \begin{cases}
     {\varphi}_n'(t,x)+\theta r_n^+(x,\phi,\psi_{n}^{t}){\varphi}_n(t,x)+\int_{S}{\varphi}_n(t,y)q_{n}(dy|x,\phi,\psi_{n}^{t})\le 0,    x\in S_{n},   \\
     {\varphi}_n'(t,x)=0, x\notin S_{n},   \\
       {\varphi}_n(T,x)=e^{\theta g_{n}^{+}(x)}, x\in S_{n}, \\
     {\varphi}_n(T,x)=1,x\notin S_{n}, \label{g2a}
    \end{cases}
\end{eqnarray}
for each $\phi\in P(A(x))$ and $t\in E_{({\varphi}_n,x)}$ with $m_L(E_{({\varphi}_n,x)}^c)=0$.

Also,  by Assumption 3.1(iii) and  (\ref{app1a}),  we have
$0\le r_{n}^{+}(x,a,b)\le r(x,a,b)\le M_{0}+\frac{\sqrt{2}}{2}\sqrt{\ln V_{0}(x)}$
for all $(x,a,b)\in K$ and $n\geq 1$. Then, using Lemma \ref{lem3.1} and Proposition \ref{P-4.1}(c) with $V_0=V_1\equiv 1$, from
(\ref{g2})  we have
\begin{eqnarray}
   e^{ -\theta \left[ T e^{\rho_{0}T}+M_{0}T+ e^{\rho_{0}T}+M_{0} \right]V_{0}(x) }\leq {\varphi}_n(t,x)=J_n(\phi_{n},\psi_{n},t,x)\leq LV_0(x), \quad \forall \
n\geq 1.  \label{u1}
\end{eqnarray}
Moreover,  ${\varphi}_n(t,x)\geq 0$. From (\ref{app1}), (\ref{app1a}) and (\ref{app1ag}),  we have $r_n^+(x,a,b)\geq r_{n-1}^+(x,a,b)$ for all $(x,a,b)\in K$,
  $q_{n-1}(dy|x,a,b)=q(dy|x,a,b)=q_{n}(dy|x,a,b)$ for $x\in S_{n-1}$,
  $q_{n-1}(dy|x,a,b)=0=q_{n}(dy|x,a,b)$ for $x\in S\setminus S_{n},$  and
  $0=q_{n-1}(dy|x,a,b)$ for $x\in S_{n}\setminus S_{n-1}$,  $g_{n-1}^{+}(x)\le g_{n}^{+}(x)$ for $x\in S_{n-1}$,
  $g_{n-1}^{+}(x)=0=g_{n}^{+}(x)$ for $x\in S\setminus S_{n},$  and
  $0=g_{n-1}^{+}(x)\le g_{n}^{+}(x)$ for $x\in S_{n}\setminus S_{n-1}$.  By (\ref{g2a}) and Proposition \ref{P-4.1}(b),(e), we have,   for all $t\in E_{({\varphi}_n,x)}$ and $n\geq 2$,
   \begin{eqnarray*}
   \begin{cases}
    {\varphi}_n'(t,x)+ \theta r_{n-1}^+(x,\phi,\psi_{n}^{t}){\varphi}_n(t,x)
     +\int_{S}{\varphi}_n(t,y)q_{n-1}(dy|x,\phi,\psi_{n}^{t})\leq 0,  \  x\in S_{n-1}, &   \\
    {\varphi}_n'(t,x)+ \theta r_{n-1}^+(x,\phi,\psi_{n}^{t}){\varphi}_n(t,x)
         +\int_{S}{\varphi}_n(t,y)q_{n-1}(dy|x,\phi,\psi_{n}^{t})= {\varphi}_n'(t,x)\leq 0,  \  x\in S_{n}\setminus S_{n-1}, & \\
    {\varphi}_n'(t,x)= 0,  \  x\in S\setminus S_{n}, & \\
          {\varphi}_n(T,x)=e^{\theta g_{n}^{+}(x)}, x\in S. \\
    \end{cases}
   \end{eqnarray*}
   Therefore,
      \begin{eqnarray*}
   \begin{cases}
    {\varphi}_n'(t,x)+ \theta r_{n-1}^+(x,\phi,\psi_{n}^{t}){\varphi}_n(t,x)
     +\int_{S}{\varphi}_n(t,y)q_{n-1}(dy|x,\phi,\psi_{n}^{t})\leq 0,  \  x\in S, &   \\
     {\varphi}_n(T,x)=e^{\theta g_{n}^{+}(x)}, x\in S.
    \end{cases}
   \end{eqnarray*}

From the proof of Proposition \ref{P-4.1}(c), we have
\begin{eqnarray*}
 &&\mathbb{E}_{\gamma}^{\pi_1,\psi_{n}} \left(e^{\int_t^T\int_{A(x_v)}\theta r_{n-1}^{+}(x_v,a,\psi_{n}^{v})\pi_1(da|x_{v},v)dv+\theta g_{n}^{+}(x_{T})}|x_{t}=x\right)-{\varphi}_{n}(t,x) \nonumber \\
  &=&\mathbb{E}_{\gamma}^{\pi_1,\psi_{n}} \left(e^{\int_t^T\int_{A(x_v)}\theta r_{n-1}^{+}(x_v,a,\psi_{n}^{v})\pi_1(da|x_v,v)dv}{\varphi}_{n}(T,x_T)|x_{t}=x\right)-{\varphi}_{n}(t,x)\leq 0.
 \end{eqnarray*}
 Therefore, by (\ref{app1ag}), we have
 \begin{eqnarray}
&&J_{n-1}(\pi_{1},\psi_{n},t,x)\nonumber\\
&=&\mathbb{E}_{\gamma}^{\pi_1,\psi_{n}} \left(e^{\int_t^T\int_{A(x_v)}\theta r_{n-1}^{+}(x_v,a,\psi_{n}^{v})\pi_1(da|x_{v},v)dv+\theta g_{n-1}^{+}(x_{T})}|x_{t}=x\right) \nonumber\\
&\leq&
\mathbb{E}_{\gamma}^{\pi_1,\psi_{n}} \left(e^{\int_t^T\int_{A(x_v)}\theta r_{n-1}^{+}(x_v,a,\psi_{n}^{v})\pi_1(da|x_{v},v)dv+\theta g_{n}^{+}(x_{T})}|x_{t}=x\right)\nonumber\\
&\le &{\varphi}_{n}(t,x) \quad {\rm for
\  all \ } \pi_{1}\in \Pi_{1}^{m}, \ (t\times x)\in [0,T]\times S. \label{K4}
\end{eqnarray}
}
Further, from (\ref{K4}), we obtain
\begin{eqnarray*}
\sup_{\pi_{1}\in \Pi_{1}^{m}}J_{n-1}(\pi_{1},\psi_{n},t,x)\leq {\varphi}_{n}(t,x) \ \ \ \ {\rm for
\  all \ }  \ t\in [0,T],x\in S.
\end{eqnarray*}
Therefore, $${\varphi}_{n-1}(t,x)=\sup_{\pi_{1}\in \Pi_{1}^{m}}\inf_{\pi_{2}\in \Pi_{2}^{m}}J_{n-1}(\pi_{1},\pi_{2} ,t,x)\le
\sup_{\pi_{1}\in \Pi_{1}^{m}}J_{n-1}(\pi_{1},\psi_{n},t,x) \leq {\varphi}_n(t,x).$$
   Thus, we have ${\varphi}_{n-1}(t,x)\leq {\varphi}_n(t,x)$, that is, the sequence $\{{\varphi}_n, n\geq 1\}$ is nondecreasing in $n\geq 1$, and thus the limit
   \begin{eqnarray}
   {\varphi}(t,x):=\lim_{n\to\infty}{\varphi}_n(t,x)\label{PG} \end{eqnarray} exists for each $(t,x)\in [0,T]\times S.$

For every $n\geq 1$ and $(t,x)\in [0,T]\times S$, we define
\begin{eqnarray}
&& H_{n}(t,x):=\sup_{\phi\in P(A(x))}\inf_{\psi\in P(B(x))}\left[\theta r_{n}^+(x,\phi,\psi){\varphi}_{n}(t,x)+\int_S{\varphi}_{n}(t,y)q_{n}(dy|x,\phi,\psi)\right],\label{Hn}\\
&& H(t,x):=\sup_{\phi\in P(A(x))}\inf_{\psi\in P(B(x))}\left[\theta r(x,\phi,\psi){\varphi}(t,x)+\int_S{\varphi}(t,y)q(dy|x,\phi,\psi)\right].\nonumber
\end{eqnarray}
We next show that $\lim_{n\to\infty} H_{n}(t,x)=H(t,x)$ for each
$(t,x)\in [0,T]\times S$.

From Assumption \ref{ass:3.3}, there exist $\phi_{n}^{t}\in P(A(x))$ and  $\psi_{n}^{t}\in P(B(x))$ such that
\begin{eqnarray*}
 H_{n}(t,x)=\theta r_{n}^+(x,\phi_{n}^{t},\psi_{n}^{t}){\varphi}_{n}(t,x)+\int_S{\varphi}_{n}(t,y)q_{n}(dy|x,\phi_{n}^{t},\psi_{n}^{t}).
\end{eqnarray*}
By (\ref{u1}), we have  $\limsup_{n\to
\infty}H_{n}(t,x)=\lim_{m\to \infty}H_{n_m}(t,x)$ for some
subsequence $\{n_m, m\geq1\}$ of $\{n, n\geq1 \}$.
For each $m\geq 1$, under Assumption \ref{ass:3.3}, the measurable
selection theorem (e.g. Proposition 7.50 in \cite{BS96}) together
with Lemma 4.1 in \cite{GuoZhang20}  ensures the existence of $\phi_{n_m}\in
\Pi_1^{M}$ and $\psi_{n_m}\in \Pi_2^{M}$ such that
\begin{eqnarray}
\hspace{-2cm}&&H_{n_m}(t,x)=\sup_{\phi\in P(A(x))}\inf_{\psi\in P(B(x))}\left[\theta r_{n_m}^+(x,\phi,\psi){\varphi}_{n_m}(t,x)+\int_{S}{\varphi}_{n_m}(t,y)q_{n_m}(dy|x,\phi,\psi)\right] \nonumber \\
\hspace{-1cm}&=&\theta
r_{n_m}^+(x,\phi_{n_m}^{t},\psi_{n_m}^{t}){\varphi}_{n_m}(t,x)+\int_{S}{\varphi}_{n_m}(t,y)q_{n_m}(dy|x,\phi_{n_m}^{t},\psi_{n_m}^{t}).
\label{H2s}
\end{eqnarray}
Since $\phi_{n_m}^{t}\in P(A(x))$ for all $m\geq 1$ and $P(A(x))$
is compact, there exists a subsequence $\{\phi_{n_{m_k}}^{t}, k\geq
1\}$ of $\{\phi_{n_m}^{t}, m\geq 1\}$ and $\bar{\phi}^{t}\in
P(A(x))$ (depending on $(t,x)$)  such that $\phi_{n_{m_k}}^{t}\to
\bar{\phi}^{t}$ as $k\to\infty$. It is the same for
$\psi_{n_m}^{t}$. Thus, $\lim_{m\to \infty}H_{n_m}(t,x)=\lim_{k\to
\infty}H_{n_{m_k}}(t,x)$. Indeed, for any fixed $(t,x)\in
[0,T]\times S$,  there exists $n_0\geq 1$ such that  $(t,x)\in
[0,T]\times S_{n_0}$, and then $q_{n}(dy|x,a,b)=q(dy|x,a,b)$ for all
$n\geq n_0$ and $\lim_{n\to\infty} r_n^+(x,a,b)=r(x,a,b)$ for all
$a\in A(x), b\in B(x)$. Thus, by Lemma 8.3.7 in \cite{one99} and
(\ref{H2s}), (\ref{Hn}) and Assumption \ref{ass:3.3}, we have
\begin{eqnarray*}
&&\limsup_{n\to \infty}H_{n}(t,x)= \lim_{k\to
\infty}H_{n_{m_k}}(t,x)\\
&=& \lim_{k\to
\infty} \theta
r_{n_{m_k}}^+(x,\phi_{n_{m_k}}^{t},\psi_{n_{m_k}}^{t}){\varphi}_{n_{m_k}}(t,x)+\int_{S}{\varphi}_{n_{m_k}}(t,y)
q_{n_{m_k}}(dy|x,\phi_{n_{m_k}}^{t},\psi_{n_{m_k}}^{t})\\
&=& \lim_{k\to
\infty}  \inf_{\psi\in P(B(x))}\left[\theta r_{n_{m_k}}^+(x,\phi_{n_{m_k}}^{t},\psi){\varphi}_{n_{m_k}}(t,x)+\int_{S}{\varphi}_{n_{m_k}}(t,y)
q_{n_{m_k}}(dy|x,\phi_{n_{m_k}}^{t},\psi)\right]\\
&\le & \lim_{k\to \infty}\left[\theta
r_{n_{m_k}}^{+}(x,\phi_{n_{m_k}}^{t},\psi){\varphi}_{n_{m_k}}(t,x)+\int_{S}{\varphi}_{n_{m_{k}}}(t,y)q_{n_{m_{k}}}
          (dy|x,\phi_{n_{m_{k}}}^{t},\psi)\right] \\
          &=& \theta r(x,\bar{\phi}^{t},\psi){\varphi}(t,x)+\int_{S}{\varphi}(t,y)q(dy|x,\bar{\phi}^{t},\psi),\quad \forall \psi\in P(B(x)).
\end{eqnarray*}

%
%
Hence, \begin{eqnarray}
 && \limsup_{n\to \infty}H_{n}(t,x)\leq
\inf_{\psi\in P(B(x))} \left[\theta r(x,\bar{\phi}^{t},\psi){\varphi}(t,x)+\int_{S}{\varphi}(t,y)q(dy|x,\bar{\phi}^{t},\psi)\right]\nonumber\\
&\leq& \sup_{\phi\in P(A(x))}\inf_{\psi\in P(B(x))} \left[\theta
r(x,\phi,\psi){\varphi}(t,x)+\int_{S}{\varphi}(t,y)q(dy|x,\phi,\psi)\right]=H(t,x).\label{H1}
\end{eqnarray}

By the similar reasoning, we have
 \begin{eqnarray}
 && \liminf_{n\to \infty}H_{n}(t,x)\geq H(t,x).\label{H11}
\end{eqnarray}
(\ref{H1}) together with (\ref{H11}) implies that $\lim_{n\to\infty}
H_{n}(t,x)=H(t,x)$. Thus, by (\ref{g1a}) we have
 \begin{eqnarray}
     {\varphi}(t,x)={e^{\theta g(x)}}+\int_{t}^{T}\sup_{\phi\in P(A(x))}\inf_{\psi\in P(B(x))}\left[\theta r(x,\phi,\psi){\varphi}(s,x)+\int_{S}{\varphi}(s,y)q(dy|x,\phi,\psi)\right]ds. \label{3h}
\end{eqnarray}

Since ${\varphi}(t,x)$ is the integral of a measurable function, it is
an absolutely continuous function, and so ${\varphi}(t,x)$ is differential in a.e. $t\in [0,T]$ (for each fixed
$x\in S$). We can verify that  ${\varphi}(t,x)$   satisfies
(\ref{EEE-2}). To show   ${\varphi}(t,x)\in
\mathbb{B}_{V_0,V_1}^{1}([0,T]\times S)$, since  ${\varphi}(t,x)\in
\mathbb{B}_{V_0}([0,T]\times S)$ (by (\ref{u1}), (\ref{PG})), the
rest verifies that  ${\varphi}'(t,x)$ is $V_1$-bounded. Indeed, since
$\theta T| r(x,a,b)|\leq e^{2T\theta|r(x,a,b)|}\leq
e^{2T\theta(M_{0}+T\theta)}V_0(x)$, from (\ref{3h}) we have
\begin{eqnarray*}
 |{\varphi}'(t,x)| &\leq & \frac {e^{2T\theta(M_{0}+T\theta)}}{T}\|{\varphi}\|_{V_0}V_0(x)V_0(x)+ \|{\varphi}\|_{V_0}\left[\rho_0V_0(x)+2V_0(x)q^*(x)\right] \\
                  &\leq & \|{\varphi}\|_{V_0}\left[ \frac{e^{2T\theta(M_{0}+T\theta)}}{T}V_0^2(x)+\rho_0V_0(x)+2L_{0}V_0^2(x)\right] \\
                   &\leq &  \|{\varphi}\|_{V_0}M_1\left [\frac{e^{2T\theta(M_{0}+T\theta)}}{T}+\rho_0+2L_0\right]V_1(x),
\end{eqnarray*}
which implies that ${\varphi}(t,x)$ is in
$\mathbb{B}_{V_0,V_1}^{1}(\times [0,T]\times S)$, and thus (a) is
proved.
\end{proof}

\section{Proof of Theorem \ref{lem4.1} }

\begin{proof}
(a) 
 From (\ref{nn+1}) and (\ref{Bf}),  we have,  $\forall  n\geq 0,
(t,x)\in [0,T]\times S,$
\begin{eqnarray}
|v_{n+1}(t,x)-v_{n}(t,x)|&=& |\Gamma^{n}v_{1}(t,x)-\Gamma^nv_{0}(t,x)|\nonumber\\
 &\leq &
\frac{\left(\theta\|r\|+2\|q\|\right)^n(T-t)^n}{n!}\|v_{1}-v_{0}\| \nonumber\\
& \leq &
\frac{\left(\theta\|r\|+2\|q\|\right)^nT^n}{n!}\|v_{1}-v_{0}\| \label{vn}
.
\end{eqnarray}
Since $\sum_{n=0}^{\infty}\frac{\left(\theta\|r\|+2\|q\|\right)^nT^n}{n!}=e^{\left(\theta\|r\|+2\|q\|\right)T}$,

$$\lim_{n\to \infty}\frac{\left(\theta\|r\|+2\|q\|\right)^nT^n}{n!}=0.$$
We also obtain $\frac{\left(\theta\|r\|+2\|q\|\right)^nT^n}{n!}$ monotonically decreases when  $n$ increases and $n>(\theta\|r\|+2\|q\|)T$.
Therefore, given $\varepsilon>0$, there exists $N_{1}>0$ such that, for all $n\geq N_{1}$,
\begin{eqnarray*}
\|v_{n+1}-v_{n}\|&=&\sup_{(t,x)\in [0,T]\times S}\|v_{n+1}(t,x)-v_{n}(t,x)\|\nonumber\\
&\le & \frac{\left(\theta\|r\|+2\|q\|\right)^nT^n}{n!}\|v_{1}-v_{0}\| <\frac{\varepsilon}{2e^{\left(\theta\|r\|+2\|q\|\right)T}\left(1+\frac{2\|q\|}{\theta \|r\|}\right)}. 
\end{eqnarray*}

Now, we prove that the sequence $\{v_{n}(t,x),n=0,1,2,\ldots\}$ converges.
From (\ref{vn}) and (\ref{14}),  for all $n\geq N_{1}, m\ge 1$,  $\forall \
(t,x)\in [0,T]\times S,$
\begin{eqnarray}
|v_{n+m}(t,x)-v_{n+1}(t,x)|
&\leq&      \sum\limits_{p=1}^{m-1} |v_{n+(p+1)}(t,x)-v_{n+p}(t,x)|\nonumber\\
&\le &  \sum\limits_{p=1}^{m-1} \frac{\left(\theta\|r\|+2\|q\|\right)^{n+p}T^{n+p}}{(n+p)!}\|v_{1}-v_{0}\|\nonumber\\
&\le & \frac{\left(\theta\|r\|+2\|q\|\right)^{n}T^{n}}{n!}  \sum\limits_{p=1}^{m-1}
 \frac{\left(\theta\|r\|+2\|q\|\right)^{p}T^{p}}{p!}\|v_{1}-v_{0}\|\nonumber\\
 &<& \frac{\left(\theta\|r\|+2\|q\|\right)^{n}T^{n}}{n!}   e^{\left(\theta\|r\|+2\|q\|\right)T}\|v_{1}-v_{0}\|<\frac{\varepsilon}{2}. \label{n+m}
\end{eqnarray}
From the above formula we can see the sequence $\{v_{n}(t,x),n=0,1,2,\ldots\}$ is a Cauchy sequence. Since $\mathbb{B}^{1}_{1,1}([0,T]\times S)$ is a Banach space,
the sequence $\{v_{n}(t,x),n=0,1,2,\ldots\}$  has a limit. From Proposition \ref{P-4.1}, we obtain  $\lim_{n\to \infty}v_n(t,x)=\varphi(t,x)$, for any $(t,x)\in [0,T]\times S$.
By (\ref{n+m}) and let $m\to \infty$, we obtain
\begin{equation*}
\|\varphi-v_{n+1}\|=\sup_{(t,x)\in [0,T]\times S}|\varphi(t,x)-v_{n+1}(t,x)|<\frac{\varepsilon}{2}, \quad \forall n\ge N_{1}.
\end{equation*}

(b) By Assumption \ref{ass:3.3}, the compactness of $P(A(x))$ and $P(B(x))$ and the continuity of $r,q,v$,  there exist  Markov policies $\phi_{n}\in \Pi_1^m, \psi_{n} \in
\Pi_2^m$ such
that (\ref{15}) holds. That is, $\forall n=0,1,2,\ldots$,  $\forall (t,x)\in [0,T]\times S$, by (\ref{nn+1}),
\begin{eqnarray}
&&v_{n+1}(t,x)\nonumber\\
&=&e^{\theta g(x)}+
\int_t^T\sup_{\phi\in P(A(x))}\inf_{\psi\in P(B(x))}\left[\theta
r(x,\phi,\psi)v_n(s,x)+\int_{S}v_n(s,y)q(dy|x,\phi,\psi)\right]ds\nonumber\\
&=&e^{\theta g(x)}+
\int_t^T\sup_{\phi\in P(A(x))}\left[\theta
r(x,\phi,\psi_{n}^s)v_n(s,x)+\int_{S}v_n(s,y)q(dy|x,\phi,\psi_{n}^{s})\right]ds\label{vn+1s}\\
&=&e^{\theta g(x)}+\int_t^T\left[\theta r(x,\phi_{n}^s,\psi_{n}^s)v_n(s,x)+\int_{S}v_n(s,y)q(dy|x,\phi_{n}^s,\psi_{n}^s)\right]ds\label{vn+1}.
\end{eqnarray}

By (\ref{vn+1}) and (\ref{vn+1s}), for any $\pi_{1}\in \Pi_{1}^{m}$, for each $x\in S$ and $t\in E_{(\varphi,x)}$,   we have
\begin{eqnarray}&&-v_{n+1}'(t,x)\nonumber\\&=&\theta r(x,\phi_{n}^t,\psi_{n}^t)v_n(t,x)+\int_{S}v_n(t,y)q(dy|x,\phi_{n}^t,\psi_{n}^t)\nonumber\\
&\ge &\theta r(x,\pi_1^t,\psi_{n}^t) v_{n}(t,x)+\int_{S}v_{n}(t,y)q(dy|x,\pi_1^t,\psi_{n}^t). \label{vn+1ge}
\end{eqnarray}

Together with the fact that $x_t(\omega)$ is a piece-wise constant (by (\ref{T_k})), (\ref{vn+1ge})  implies
\begin{eqnarray}
 && -\left( e^{\int_0^t\theta r(x_v,\pi_1^v,\psi_{n}^v)dv}v_{n+1}(t,x_t) \right)'\nonumber\\
&=& - e^{\int_0^t\theta r(x_v,\pi_1^v,\psi_{n}^v)dv}\left[
\theta r(x_t,\pi_1^t,\psi_{n}^t)v_{n+1}(t,x_t)+v_{n+1}'(t,x_t) \right]\nonumber\\
&\ge & - e^{\int_0^t\theta r(x_v,\pi_1^v,\psi_{n}^v)dv}
\theta r(x_t,\pi_1^t,\psi_{n}^t)v_{n+1}(t,x_t)+e^{\int_0^t\theta r(x_v,\pi_1^v,\psi_{n}^v)dv}\nonumber\\
&&\left[\theta v_{n}(t,x_{t}) r(x_{t},\pi_1^t,\psi_{n}^t)+\int_{S}v_{n}(t,y)q(dy|x_{t},\pi_1^t,\psi_{n}^t)\right]\nonumber \\
    &= &e^{\int_0^t\theta r(x_v,\pi_1^v,\psi_{n}^v)dv}\int_{S}v_{n}(t,y)q(dy|x_t,\pi_1^t,\psi_{n}^t) \nonumber \\
   &&  + e^{\int_0^t\theta r(x_v,\pi_1^v,\psi_{n}^v)dv}
\theta r(x_t,\pi_1^t,\psi_{n}^t)\left[v_{n}(t,x_t)-v_{n+1}(t,x_t)\right]\nonumber\\
 &\ge &e^{\int_0^t\theta r(x_v,\pi_1^v,\psi_{n}^v)dv}\int_{S}v_{n}(t,y)q(dy|x_t,\pi_1^t,\psi_{n}^t) -e^{t\theta \|r\|}
\theta \|r\|\|v_{n+1}-v_{n}\|,
 \ \ \forall \ t\ge 0. \label{ineq}
\end{eqnarray}

 Thus, by Lemma \ref{Th-3.2}(b) and (\ref{ineq}), we have

 \begin{eqnarray*}
&&\mathbb{E}_{\gamma}^{\pi_1,\psi_{n}}\left[e^{\int_s^T \theta
r(x_t,\pi_1^t,\psi_{n}^t)dt}v_{n+1}(T,x_{T})\Big|x_s=x\right]-v_{n+1}(s,x) \\
&=&\mathbb{E}_{\gamma}^{\pi_1,\psi_{n}}\left\{\int_{s}^{T}
\left[\left(e^{\int_s^t\theta
r(x_v,\pi_1^v,\psi_{n}^v)dv}v_{n+1}(t,x_t)\right)'\right.\right.\\
&&\left.\left.+\int_Se^{\int_s^t\theta
r(x_v,\pi_1^v,\psi_{n}^v)dv}v_{n+1}(t,y)q(dy|x_t,\pi_1^t,\psi_{n}^t)\right]dt\Big|x_s=x\right\}\\
&\le &\mathbb{E}_{\gamma}^{\pi_1,\psi_{n}}\left\{\int_{s}^{T}
\left[-\int_{S}q(dy|x_t,\pi_1^t,\psi_{n}^t)e^{\int_s^t\theta r(x_v,\pi_1^v,\psi_{n}^v)dv}v_{n}(t,y) \right.\right.\nonumber\\
   &&  +e^{t\theta \|r\|}
\theta \|r\|\|v_{n+1}-v_{n}\|\left.\left.+\int_Se^{\int_s^t\theta
r(x_v,\pi_1^v,\psi_{n}^v)dv}v_{n+1}(t,y)q(dy|x_t,\pi_1^t,\psi_{n}^t)\right]dt\Big|x_s=x\right\}\\
&= &\mathbb{E}_{\gamma}^{\pi_1,\psi_{n}}\left\{\int_{s}^{T}
\left[\int_{S}q(dy|x_t,\pi_1^t,\psi_{n}^t)e^{\int_s^t\theta r(x_v,\pi_1^v,\psi_{n}^v)dv}\left(v_{n+1}(t,y)-v_{n}(t,y)\right) \right.\right.\nonumber\\
   && \left.\left. +e^{t\theta \|r\|}
\theta \|r\|\|v_{n+1}-v_{n}\|\right]dt\Big|x_s=x\right\}\\
&\le  &e^{T\theta \|r\|}
\|v_{n+1}-v_{n}\|+2\|q\|\|v_{n+1}-v_{n}\|\int_{0}^{T}e^{t\theta \|r\|}dt\\
&\le  &e^{T\theta \|r\|}
\|v_{n+1}-v_{n}\|+2\|q\|\|v_{n+1}-v_{n}\|\frac{1}{\theta \|r\|}e^{T\theta \|r\|}\\
&\le  &e^{T\theta \|r\|}
\|v_{n+1}-v_{n}\|\left(1+\frac{2\|q\|}{\theta \|r\|}\right),\ \ \ \ \forall \ s\in [0,T].
\end{eqnarray*}

  For all $n\geq N_{1}$, by (\ref{14}),
\begin{eqnarray*}
J(\pi_1,\psi_{n},s,x)-v_{n+1}(s,x)<\frac{\varepsilon}{2} \quad  {\rm for
\  all} \ (s,x)\in [0,T]\times S, \  {\rm for \  all}\ \pi_1\in\Pi_1.
\end{eqnarray*}
Therefore,
\begin{equation}\label{eq:le}
J(\phi_{n},\psi_{n},t,x)-v_{n+1}(t,x)<\frac{\varepsilon}{2}, \quad \forall (t,x)\in [0,T]\times S.
\end{equation}
A similar proof gives
\begin{equation}\label{eq:ge}
J(\phi_{n},\psi_{n},t,x)-v_{n+1}(t,x)>-\frac{\varepsilon}{2}, \quad \forall (t,x)\in [0,T]\times S.
\end{equation}
Thus, by (\ref{eq:le}) and (\ref{eq:ge}), we get
\begin{equation}\label{eq:ab}
|J(\phi_{n},\psi_{n},t,x)-v_{n+1}(t,x)|<\frac{\varepsilon}{2}, \quad \forall (t,x)\in [0,T]\times S.
\end{equation}
From (\ref{n+1}) and (\ref{eq:ab}),
\begin{eqnarray*}
&&\sup_{(t,x)\in [0,T]\times S}|J(\phi_{n},\psi_{n},t,x)-\varphi(t,x)|\\&=&\sup_{(t,x)\in [0,T]\times S}
|J(\phi_{n},\psi_{n},t,x)-v_{n+1}(t,x)+v_{n+1}(t,x)-\varphi(t,x)|\\
&\le &\sup_{(t,x)\in [0,T]\times S}
\left(|J(\phi_{n},\psi_{n},t,x)-v_{n+1}(t,x)|+|v_{n+1}(t,x)-\varphi(t,x)|\right)<\varepsilon.
\end{eqnarray*}

\end{proof}

\end{appendices}

\end{document}